\documentclass[a4paper,fleqn,reqno]{amsart}

\usepackage{amsmath,amsfonts,amsthm,amssymb,amsxtra,enumerate,color,times,euscript, latexsym}

\def\sS{{\mathfrak S}}

      \def\dC{{\mathbb C}}

   \def\dN{{\mathbb N}}

   \def\dZ{{\mathbb Z}}

   \def\cQ{{\mathcal Q}}

\newcommand{\ar}{\alpha^R}
\newcommand{\al}{\alpha^L}

\def\dist{\operatorname{dist}}

\def\downbar#1{
\setbox10=\hbox{$#1$}
   \dimen10=\ht10 \advance\dimen10 by 2.5pt
   \ifdim \dimen10<15pt 
      \advance\dimen10 by -0.5pt
      \dimen11=\dimen10
      \advance\dimen10 by 2.5pt
      \lower \dimen11
   \else \lower \ht10 \fi
   \hbox {\hskip 1.5pt \vrule height \dimen10 depth \dp10}\relax}
 \def\upbar#1{
 \setbox10=\hbox{$#1$}
    \dimen10=\ht10 \advance\dimen10 by \dp10 \advance\dimen10 by 2.5pt
    \ifdim \dimen10<15pt 
       \advance\dimen10 by 2pt \fi
    \raise 2.5pt \hbox {\hskip -1.5pt \vrule height \dimen10}\relax}
\def\cfr#1#2{
 \downbar{#2} \hskip -1.5pt {\; #1 \; \over \thinspace \  #2}\upbar{#1}}

\newcommand{\be}{\begin{equation}}

\newcommand{\ee}{\end{equation}}

\newcommand{\ba}{\begin{eqnarray}}
\newcommand{\ea}{\end{eqnarray}}

\newcommand{\baa}{\begin{eqnarray*}}
\newcommand{\eaa}{\end{eqnarray*}}
\newcommand{\bb}{\begin{thebibliography}}
\newcommand{\eb}{\end{thebibliography}}

\newcommand{\re}[1]{(\ref{#1})}

\def\Im{\operatorname{Im}}
\def\Re{\operatorname{Re}}

\newtheorem{theorem}{Theorem}[section]
\newtheorem{proposition}[theorem]{Proposition}
\newtheorem{corollary}[theorem]{Corollary}
\newtheorem{lemma}[theorem]{Lemma}
\newtheorem{definition}[theorem]{Definition}
\theoremstyle{definition}
\newtheorem{example}[theorem]{Example}
\newtheorem{remark}[theorem]{Remark}

\numberwithin{equation}{section}

\begin{document}
\title[The linear pencil approach to rational interpolation]
{The linear pencil approach to rational interpolation}

\author{Bernhard Beckermann}
\address{Bernhard Beckermann\\ Laboratoire Painlev\'e UMR 8524 (ANO-EDP), UFR Math\'ematiques --
M3\\ UST Lille, F-59655 Villeneuve d'Ascq CEDEX, France}
\email{bbecker@math.univ-lille1.fr}

\author{Maxim  Derevyagin}
\address{Maxim  Derevyagin\\ Department of Nonlinear Analysis \\
Institute of Applied Mathematics and Mechanics\\ R.Luxemburg str.
74 \\ 83114 Donetsk, Ukraine, and 
Laboratoire Painlev\'e, UFR Math\'ematiques --
M3\\ UST Lille, F-59655 Villeneuve d'Ascq CEDEX, France\\}
\email{derevyagin.m@gmail.com}

\author{Alexei Zhedanov}

\address{Alexei Zhedanov\\ Institute for Physics and Engineering\\
R.Luxemburg str. 72 \\
83114 Donetsk, Ukraine}

\email{zhedanov@yahoo.com}


\begin{abstract}
     It is possible to generalize the
     fruitful interaction between (real or complex) Jacobi
     matrices, orthogonal polynomials and Pad\'e approximants at infinity
     by considering rational interpolants, (bi-)orthogonal
     rational functions and linear pencils $zB-A$ of two
     tridiagonal matrices $A,B$, following Spiridonov and Zhedanov.

     In the present paper, beside revisiting the underlying
     generalized Favard theorem, we suggest a new criterion for
     the resolvent set of this linear pencil in terms of the
     underlying  associated rational functions. This enables us to
     generalize several convergence results for Pad\'e approximants
     in terms of complex Jacobi matrices to the more general case
     of convergence of rational interpolants in terms of the
     linear pencil. We also study generalizations of the
     Darboux transformations and the link to
     biorthogonal rational functions. Finally, for a Markov function
     and for pairwise
     conjugate interpolation points tending to $\infty$, we
     compute explicitly the spectrum and the numerical range of
     the underlying linear pencil.
\end{abstract}



\maketitle

\section{Introduction}\label{sec_1}

The connection with Jacobi matrices has led to numerous
applications of spectral techniques for self-adjoint operators in
the theory of orthogonal polynomials on the real line and Pad\'e
approximation. In order to give an idea of these interactions
consider a Markov function of the form
\[
   \varphi(z)=\int_{a}^{b}\frac{d\mu(t)}{z-t},
\]
where $a,b$ are real numbers and $d\mu(t)$ is a probability
measure, that is, $\int_a^b d\mu(t)=1$. It is well
known~\cite{A},~\cite{wall} that one can expand such a Markov
function $\varphi$ into the following continued fraction
\begin{equation}\label{Jfraction}
\varphi(z)=\frac{1}{z-b_0-\displaystyle{\frac{a_0^2}{z-b_1-\displaystyle{\frac{a_1^2}{\ddots}}}}}=
\cfr{1}{z-b_0} -\cfr{a_0^2}{z-b_1} -\cfr{a_{1}^2}{z-b_2} -\cdots,
\end{equation}
where $b_j,a_j\in \mathbb R$, $a_j>0$. Continued fractions of the
form~\eqref{Jfraction} are called J-fractions~\cite{JT,wall}. To
the continued fraction~\eqref{Jfraction} one can associate a
Jacobi matrix $A$ acting in the space of square summable
sequences and its truncation $A_{[0:n]}$
\[
A=\left(%
\begin{array}{cccc}
  b_0 & a_0 &  &  \\
  a_0 & b_1 & a_1 &  \\
      & a_1 & b_2 & \ddots \\
      &     & \ddots & \ddots \\
\end{array}%
\right),\quad
A_{[0:n-1]}=\left(%
\begin{array}{cccc}
  b_0 & a_0    &  &  \\
  a_0 & b_1    & \ddots &  \\
      & \ddots & \ddots & a_{n-2} \\
      &        & a_{n-2}& b_{n-1}     \\
\end{array}%
\right).
\]
Then it is known that $\varphi(z)=\langle
(zI-A)^{-1}e_0,e_0\rangle$,  and the $n$th convergent of the above
continued fraction is given by
\[
   \frac{p_n(z)}{q_n(z)}=\langle (zI-A_{[0:n-1]})^{-1}e_0,e_0\rangle=
\cfr{1}{z-a_0}- \dots -\cfr{b_{n-2}^2}{z-a_{n-1}},
\]
where the column vector $e_0=(1,0,\dots)^\top$ is the first
canonical vector of suitable size, $q_n$ are orthogonal
polynomials with respect to $d\mu$, and $p_n$ are polynomials of
the second kind, see~\cite{A,NikSor,S98}. It is elementary fact of
the continued fraction theory that
\begin{equation}\label{dPA}
   \varphi(z)-\frac{p_n(z)}{q_n(z)}=O\left(\frac{1}{z^{2n+1}}\right)_{z\rightarrow\infty},
\end{equation}
see for instance~\cite{A,BGM,JT}. Relation~\eqref{dPA} means that
the rational function $p_n/q_n$ is the $n$th diagonal Pad\'e
approximant to $\varphi$ at infinity. Consequently, the locally
uniform convergence of diagonal Pad\'e approximants appears as the
strong resolvent convergence of the finite matrix approximations
$A_{[0:n]}$. For instance, one knows that $p_n/q_n\to \varphi$ in
capacity in the resolvent set of $A$ given by the complement of the
support of $\mu$, and locally uniformly outside the numerical
range of $A$ given by the convex hull of the spectrum of $A$,
see for instance~\cite{STO}. Besides, it should be mentioned here
that an operator approach for proving convergence of
Pad\'e approximants for rational perturbations of Markov functions
was proposed in~\cite{DD}, see also~\cite{D09}.

If $\varphi$ is no longer a Markov function but has distinct $n$th
diagonal Pad\'e approximants at infinity, we may still recover these
approximants as convergents of a continued fraction of type
\eqref{Jfraction}, but now in general $a_j,b_j \in \mathbb C$,
$a_j \neq 0$, see \cite{wall}, that is, $A$ becomes complex
symmetric, called a complex Jacobi matrix. There is no longer a
natural candidate for the spectrum of $A$, but it is still
possible to characterize the spectrum in terms of some asymptotic
behavior of the Pad\'e denominators $q_n(z)$ and the linearized
error functions $r_n(z)=q_n(z)\phi(z)-p_n(z)$
\cite{AptKV,BK97,B00}, see also \cite{BO03,DD,D09} for more
general banded matrices. Convergence outside the numerical range
was established in \cite{BK97}, and convergence in capacity in the
outer connected component of the resolvent set in \cite{B99}. We
refer the reader to \cite{B01} for some recent summary on complex
Jacobi matrices, including some open questions partially solved in
\cite{BC04}.

%

The goal of this paper is to generalize several of the above
results 
to the case of multipoint Pad\'e approximants.


\begin{definition}[\cite{BGM}]\label{PadeInt}
The $[n_1|n_2]$ multipoint Pad\'e approximant (or rational
interpolant) for a function $\varphi$ at the points
$\{z_k\}_{k=1}^{\infty}$ is defined as the ratio $p/q$ of two
polynomials $p$ and $q\neq 0$ of degree at most $n_1$ and $n_2$,
respectively, such that $\varphi q - p$ vanishes at
$z_1,z_2,...,z_{n_1+n_2+1}$ counting multiplicities.
\end{definition}

It is easy to see that the degree and interpolation conditions
lead to a homogeneous system of linear equations, and thus an
$[n_1|n_2]$ multipoint Pad\'e approximant exists. Also, one may
show uniqueness of the fraction $p/q$. However, since the
denominator may vanish at some of the interpolation points, it may
happen that the fraction $p/q$ does not interpolate
$\varphi$ at some point $z_k$, usually referred to as an
unattainable point.

Under some regularity conditions, the $[n-1|n]$ multipoint Pad\'e
approximants of $\varphi$ may be written as $n$th convergents of a
continued fraction of the form
\begin{equation}\label{R2}
      \cfr{1}{z-b_0}- \cfr{a_0^2(z-z_1)(z-z_2)}{z-b_1}-
      \cfr{a_1^2(z-z_3)(z-z_4)}{z-b_2}-
      \dots ,
\end{equation}
the odd part of a Thiele continued fraction \cite{BGM}. Continued
fractions of this type are referred to as $MP$-fractions in
\cite{HN1} and as $R_{II}$--fractions in \cite{IM95}. In
particular, the authors study in \cite[Theorem~~4.4]{HN1} and
\cite[Theorem~3.5]{IM95} some analog of Favard's theorem and the
link with orthogonal rational functions.
Spiridonov and one of the authors~\cite{SZ, Zhe_BRF}
showed that such continued fractions are related not to a single
Jacobi matrix but to a pencil $zB-A$ with tridiagonal matrices
$A,B$. Various links to bi-orthogonal rational functions have been
presented in \cite{Zhe_BRF} and~\cite{DZ09}. In particular,
in~\cite[Theorem~6.2]{DZ09} the authors present an
operator-theoretic proof for the Markov convergence theorem
multipoint Pad\'e approximants  \cite{GL} based on spectral
properties of the pencil $zB-A$.

The aim of this paper is to present further convergence results
for the continued fraction \eqref{R2}, both in the resolvent set
and outside the numerical range of the tridiagonal linear pencil
$zB-A$. To be more precise, denote by $\ell^2=\ell^2_{[0:\infty)}$
the Hilbert space of complex square summable sequences
$(x_0,x_1,\dots)^{\top}$ with the usual inner product
\[
\langle x,y\rangle =\sum_{j=0}^{\infty}x_j\overline{y}_j,\quad
x,y\in\ell^2.
\]
We will restrict our attention to the case of tridiagonal matrices
$A,B$ with bounded entries, in which case we may identify via
usual matrix product the matrices $A$ and $B$ with bounded
operators acting in $\ell^2$. Notice that many
algebraic relations remain true in the unbounded case as well.
However, already the simpler case of bounded pencils allows to
describe the main ideas of how to generalize results from the
classical theory of orthogonal polynomials to the theory of
biorthogonal rational functions as well as to the multipoint
Pad\'e approximation.

The remainder of the paper is organized as follows: we start from
a general bounded $MP$--fraction and introduce the associated
linear pencils together with the rational solutions of some
underlying three term recurrence relations in \S\ref{sec_2.1}. In
\S\ref{sec_2.2}, by generalizing previous work of Aptekarev,
Kaliaguine \& Van Assche \cite{AptKV} we show how the asymptotic
behavior of these rational solutions allows to decide whether the
linear pencil $zB-A$ is boundedly invertible. In particular, we
deduce in Corollary~\ref{pointwise} the pointwise convergence of
at least a subsequence of our multipoint Pad\'e approximants
towards what is called the $m$--function (or Weyl function) of the
linear pencil. Subsequently, we present in Theorem~\ref{favard} of
\S\ref{sec_2.3} an alternate proof for a Favard-type theorem
based on orthogonality properties of associated
rational functions, which yields in
Corollary~\ref{interpolants} a simple proof for the fact that the
convergents of our continued fractions are indeed multipoint Pad\'e
approximants of the $m$-function of our linear pencil.
In \S\ref{sec_3} we generalize the above-mentioned results of
\cite[Theorem~3.10]{BK97}, \cite[Theorem~3.1]{B99}, and
\cite[Theorem~4.4]{B99}, on the convergence of Pad\'e approximants
at infinity in terms of complex Jacobi matrices to the more
general case of multi-point Pad\'e approximants in terms of linear
pencils $zB-A$. The aim of \S\ref{sec_4} is to explore $LU$ and
$UL$ decompositions of our linear pencil, and the link to
biorthogonal rational functions. This naturally leads us to
consider generalizations of the Darboux
transformations of \cite{BM04}. Finally, we generalize in
\S\ref{sec_5} the findings described in the begining of this
section, namely, if we start with a Markov function and pairwise
conjugate interpolation points tending to infinity, then the
spectrum of our linear pencil is still the support of the
underlying measure, and the numerical range equals its convex hull.

\section{Continued fractions, linear pencils, and their
resolvents}

  In this section we show the links between continued
fractions in question and linear pencils. Moreover, we prove a
Favard type result for the corresponding recurrence relation.
\subsection{Linear pencils}\label{sec_2.1}
Let us consider a continued fraction of the form
\begin{equation}\label{CF}
    \cfr{1}{\beta_0(z)}-
    \cfr{\alpha_0^L(z)\alpha_0^R(z)}{\beta_1(z)}-
    \cfr{\alpha_1^L(z)\alpha_1^R(z)}{\beta_2(z)}-
    \dots
\end{equation}
where $\beta_n,\alpha_n^L,\alpha_n^R$ are polynomials of
degree at most $1$ and not identically zero. Next,
denote by $C_n(z)$ the $n$th convergent of this continued fraction
obtained by taking only the first $n$ terms in \eqref{CF}, then
the well-known theory of continued fractions tells us that
$C_n(z)=p_n(z)/q_n(z)$, where the polynomials $p_n$ of degree
$\leq n-1$ and $q_n$ of degree $\leq n$ are obtained as solutions
of the three-term recurrence relation
\begin{equation}\label{r2rec}
   y_{n+1} = \beta_n(z)y_n- \alpha^L_{n-1}(z)\alpha^R_{n-1}(z)y_{n-1}, \quad n=0,1,2,\dots
\end{equation}
by means of the initial conditions (setting $\alpha_{-1}^L=\alpha_{-1}^R=1$ for convenience)
\begin{equation}\label{r2initial}
    q_0(z)=1,\quad q_{-1}(z)=0,\quad p_{0}(z)=0,\quad p_{-1}(z)=-1.
\end{equation}
Using~\eqref{r2rec} and \eqref{r2initial} one easily verifies by
recurrence that
\begin{equation} \label{finite_det}
     q_n(z)=\det(z B_{[0:n-1]}-A_{[0:n-1]}) , \quad
     p_n(z)=\det(z B_{[1:n-1]}-A_{[1:n-1]}).
\end{equation}
By Cramer's rule, this implies the following formula for the
convergents \begin{equation} \label{cramer}
    C_n(z)=\frac{p_n(z)}{q_n(z)} = \langle
       (z B_{[0:n-1]}-A_{[0:n-1]})^{-1} e_0, e_0 \rangle.
\end{equation}
By induction, one also easily shows the Liouville-Ostrogradsky
formula (for the classical case, see~\cite[p.~9 formula~(1.15)]{A})
\begin{equation}\label{Ostrogr2}
   p_{n+1}(z)q_{n}(z) -
        p_{n}(z) q_{n+1}(z)= \prod_{k=0}^{n-1}\al_k(z)\ar_k(z), \quad
        n=0,1,2,\dots.
\end{equation}
For a complex number $\phi(z)$, the sequence defined by
\begin{equation} \label{second_kind}
    r_n(z):=\phi(z) q_n(z)-p_n(z),
\end{equation}
gives another solution of~\eqref{r2rec} with initial conditions
\begin{equation}\label{r3initial}
    r_0(z)=\phi(z),\quad r_{-1}(z)=1 .
\end{equation}
We will refer to $r_n$ as linearized error (or function of the
second kind) since, from the Pincherle Theorem~\cite[Theorem 5.7]{JT}, the continued
fraction \eqref{CF} has a limit $\phi(z)$ iff $r_n(z)$ is a
minimial solution of the recurrence relation \eqref{r2rec}.

It will be convenient to write the polynomials $\alpha^L_j$,
$\alpha^R_j$, and $\beta_j$ occurring in (\ref{CF}) in the form of the
tridiagonal infinite linear pencil
\begin{equation}\label{linpen}
   zB-A=\begin{pmatrix}
  \beta_0(z) & -\alpha^R_0(z) & 0 & 0 & \dots \\
  -\alpha^L_0(z) & \beta_1(z) & -\alpha^R_1(z) & 0 & \ddots \\
   0   & -\al_1(z) & \beta_2(z) & \alpha^R_2(z) & \ddots \\
   \vdots & \ddots & \ddots & \ddots & \ddots
\end{pmatrix}
\end{equation}
with the two tridiagonal infinite matrices
$A=(a_{i,j})_{i,j=0}^{\infty}$ and $B=(b_{i,j})_{i,j=0}^{\infty}$.
For a $J$-fraction, we obtain the linear pencil $z - A$ with a
tridiagonal matrix $A$~\cite{A} (see also~\cite{B01}). In the case
of $J$-fractions it is also known that we may write the eigenvalue
equation $A y = z y$ for some infinite column vector $y$ in terms
of normalized counterparts of the monic polynomials $q_n(z)$
(namely the corresponding orthonormal OP). Notice that the product
$Ay$ is defined for $y$ not necessarily an element of $\ell^2$,
since for each component there are only a finite number of
non-zero terms. For the linear pencil $zB-A$ we can analogously
write the similar eigenvalue equations
\begin{equation}\label{spectral}
     A q^R(z)= z B q^R(z) , \quad
     q^L(z) A = z q^L(z)B  , \quad
\end{equation}
with an infinite column vector $q^R(z)=(q_0^R(z), q_1^R(z), \dots
)^{\top}$ and an infinite row vector $q^L(z)=(q_0^L(z), q_1^L(z),
\dots )$. Here  $q_n^L(z)$ and $q_n^R(z)$ are rational functions
obtained from $q_n(z)$ by scaling with a product of linear
polynomials. Indeed, defining $q_n^R(z)$, $p_n^R(z)$, and
$r_n^R(z)$ via
\begin{equation} \label{def_R}
    q_n^R(z) = \frac{q_n(z)}{\displaystyle \prod_{k=0}^{n-1} \alpha^R_k(z)} , \quad
    p_n^R(z) = \frac{p_n(z)}{\displaystyle \prod_{k=0}^{n-1} \alpha^R_k(z)} , \quad
    r_n^R(z) = \frac{r_n(z)}{\displaystyle \prod_{k=0}^{n-1} \alpha^R_k(z)} ,
\end{equation}
leads us to three solutions of the recurrence relation
\begin{equation} \label{rec_R}
     \alpha_n^R(z) y_{n+1}^R -\beta_n(z) y_n^R +\alpha_{n-1}^L(z) y_{n-1}^R=0, \quad
        n=0,1,2,\dots.
\end{equation}
In the similar way, we see that
\begin{equation} \label{def_L}
    q_n^L(z) = \frac{q_n(z)}{\displaystyle \prod_{k=0}^{n-1} \alpha^L_k(z)} , \quad
    p_n^L(z) = \frac{p_n(z)}{\displaystyle \prod_{k=0}^{n-1} \alpha^L_k(z)} , \quad
    r_n^L(z) = \frac{r_n(z)}{\displaystyle \prod_{k=0}^{n-1} \alpha^L_k(z)} ,
\end{equation}
are three solutions of the recurrence relation
\begin{equation} \label{rec_L}
     \alpha_n^L(z) y_{n+1}^L - \beta_n(z) y_n^L+\alpha_{n-1}^R(z) y_{n-1}^L=0, \quad
        n=0,1,2,\dots .
\end{equation}
Now, it is immediate to see by taking into account the initial
conditions (\ref{r2initial}) that the identities (\ref{rec_R})
and (\ref{rec_L}) reduce to the formal spectral equations
(\ref{spectral}).

It should be also noted that we formally have
\begin{equation}\label{spectral1}
p^L(z)(zB-A)=-e_0^{\top},\quad (zB-A)p^R(z)=-e_0.
\end{equation}
\begin{remark}\label{scaling}
    There are many degrees of freedom in going from a continued
    fraction \eqref{CF} to a linear pencil $zB-A$. For instance,
    for the special case
    of $\deg \beta_n=1$ and $\deg \alpha_n^L=0=\deg \alpha_n^R$
    for all $n \geq 0$,
    the above approach leads a priori to diagonal
    $B$ and tridiagonal $A$ without any further symmetry
    properties. However, by applying an equivalence transformation
    to \eqref{CF} we can make the polynomials $\beta_n$ monic,
    implying that $B$ is the identity matrix. Moreover, we can
    choose $\alpha_n^L=\alpha_n^R$, i.e., $A$ becomes complex
    symmetric (also called a complex Jacobi matrix).
    In this case, $q_n^L=q_n^R$ are known to be the corresponding
    formal orthonormal polynomials, whereas $q_n$ is the
    associated monic counterpart. We will return to this scaling
    and normalization freedom in the last section. \qed
\end{remark}
\subsection{$m$-functions of linear pencils and the
resolvent}\label{sec_2.2}

In accordance with the Jacobi case of $B$ being the identity, we
define the resolvent set $\rho(A,B)$ of the linear pencil $zB-A$
to be the set of $z\in \mathbb C$ such that $zB-A$ has a bounded
inverse. 
The following simple example shows that $\rho(A,B)$ can be of
arbitrary shape.

\begin{example}\label{ex_resolvent}
  For $z_0,z_1,,\dots\in \mathbb C$, consider the diagonal pencil
  \[
       D_1 z -D_2 = \mbox{diag}\left(\frac{z-z_n}{1+|z_n|}\right)_{n=0,1,2,...},
  \]
  together with the tridiagonal pencil
  \[   Bz-A=U^*(D_1 z -D_2)U, \qquad
       U = \left(\begin{array}{cccccc}
       1 & -1/2 & 0 & \cdots & \cdots
       \\
       0 & 1 & -1/2 & 0 & \cdots
       \\
       0 & 0 & 1 & -1/2 & \ddots
       \\
       \vdots & \ddots & \ddots & \ddots & \ddots
            \end{array}\right)
  \]
  with bounded $A,B$. Then according to \eqref{finite_det},
  \eqref{linpen},
  \eqref{def_R}, and \eqref{def_L},
  \[
      \alpha_n^L(z)=\alpha_n^R(z)=\frac{1}{2}\frac{z-z_n}{1+|z_n|},
      \quad
      q_n^L(z)=q_n^R(z)=2^n. 
  \]
  Since $U$ is boundedly invertible, we find that
  $\rho(A,B)=\rho(D_2,D_1) = \mathbb C \setminus \Sigma$
  with $\Sigma$ the closure of $\{ z_0,z_1,\dots \}$, which could be
  any closed set of the complex plane. In particular, $\rho(A,B)$ can be empty.
  \qed
\end{example}

Notice that, if in addition
  $B$ is boundly invertible (which for instance in
  Example~\ref{ex_resolvent} is not necessarily true) then
  $\rho(A,B)=\rho(B^{-1}A)=\rho(AB^{-1})$.
However, in general we loose for the last two operators the link
with tridiagonal matrices and three-term recurrencies, and thus we
prefer to argue in terms of pencils.


Aptekarev et al.\ \cite[Theorem~1]{AptKV} showed that a bounded
tridiagonal matrix has a bounded inverse if and only if the above
solutions of the recurrencies \eqref{rec_R} and \eqref{rec_L} have
a particular asymptotic behavior. In our setting, their findings
(see also the slight improvement given in
\cite[Theorem~2.1]{BK97}) read as follows.

\begin{theorem}[\cite{AptKV}] \label{thChar}
   Suppose that $A,B$ are bounded, and consider for $z\in \mathbb C$
   the matrix $R(z)$ with entries
   \[
          R(z)_{j,k} =
          \left\{\begin{array}{cc}
              r_j^R(z) q_k^L(z) = (q_j^R(z) \phi(z)
              -p_{j}^R(z))q_k^L(z) & \mbox{if $j\geq k$,}
              \\
              q_j^R(z) r_k^L(z) = q_j^R(z) (q_k^L(z) \phi(z)
              -p_{k}^L(z))   & \mbox{if $j\leq k$.}
           \end{array}\right.
   \]
   Then $z\in \rho(A,B)$ if and only if there
   exists $\phi(z)\in \mathbb C$ and constants $\gamma(z)>0$,
   $\delta(z)\in (0,1)$
   such that \begin{equation} \label{geometric}
       |R(z)_{j,k}|\leq \gamma(z) \delta(z)^{|j-k|}, \quad
       j,k=0,1,....
   \end{equation}
   In this case, $R(z)_{j,k}=\langle (zB-A)^{-1} e_k,e_j\rangle$,
   in particular,
   $\phi(z)$ is uniquely given by
   \[
            \phi(z)=R(z)_{0,0} = \langle (zB-A)^{-1}
            e_0,e_0\rangle.
   \]
\end{theorem}

  For bounded complex Jacobi matrices ($B=I$ and $q_n^L(z)=q_n^R(z)$),
  it was shown in \cite{B00} that
  $z\in \rho(A,B)$ can be characterized only in terms of the
  asymptotic behavior of the denominators $q_n^L(z)=q_n^R(z)$.
  As we see from Example~\ref{ex_resolvent}, this is no longer
  true for bounded tridiagonal pencils.

For the sake of completeness, we will give below the ideas of the
proof of Theorem~\ref{thChar}. Let us first discuss some immediate
consequences.

\begin{remark}
   For the particular case of Jacobi matrices (that is $B=I$),
   the above formulas for the entries of the resolvent,
   also referred to as Green's functions, have been known for a
   long time, see for instance the recent
   book~\cite[Section~4.4]{simon}. Our linear
   pencil formalism also includes so-called CMV matrices occurring
   in the study of orthogonal polynomials on the unit circle,
   see~\cite[Section~4.2]{simon}, here $A,B$ are not only
   tridiagonal but in addition block-diagonal, with unitary
   blocks. Again, the formulas for the Green's functions given in
   \cite{simon} are a special case of Theorem~\ref{thChar}.
   We also refer the reader to \cite{BKLO} for recent findings for the
   special case
   of multipoint Schur functions : here the roots of $\alpha_n^L$
   and $\alpha_n^R$ are related through reflexion across the unit circle.
   \qed
\end{remark}

A basic object in Theorem~\ref{thChar} and in the rest of the
paper is following.

\begin{definition}
The function
\begin{equation}\label{Weyl1}
    m(z)=\langle (zB-A)^{-1}e_0,e_0\rangle ,\quad z\in\rho(A,B)
\end{equation}
will be called the $m$-function (or Weyl function) of the linear
pencil $zB-A$.
\end{definition}

Comparing with \eqref{cramer} we are left with the central
question whether the $m$-function $p_n(z)/q_n(z)=\langle
(zB_{[0:n-1]}-A_{[0:n-1]})^{-1} e_0,e_0\rangle$ of the finite
pencil $zB_{[0:n-1]}-A_{[0:n-1]}$ converges for $n \to \infty$ to
the $m$-function of the infinte pencil $zB-A$.

We learn from Theorem~\ref{thChar} that the linearized errors
$r_n^L(z)=R(z)_{0,n}=q_n^L(z) m(z)-p_n^L(z)$ and
$r_n^R(z)=R(z)_{n,0}=q_n^R(z) m(z)-p_n^R(z)$ tend to zero with a
geometric rate \begin{equation} \label{residual}
    \limsup_{n \to \infty} | r_n^L(z) |^{1/n} < 1 , \quad
    \limsup_{n \to \infty} | r_n^R(z) |^{1/n} < 1 , \quad
    z \in \rho(A,B) .
\end{equation}
Following exactly the lines of Aptekarev et al.\ \cite[Theorem~2
and Corollary~3]{AptKV} we obtain the following result on
point-wise convergence of a subequence.

\begin{corollary}\label{pointwise}
   We have for $z\in \rho(A,B)$
   \[
    \limsup_{n \to \infty} | q_n^L(z) |^{1/n} > 1 , \,
    \limsup_{n \to \infty} | q_n^R(z) |^{1/n} > 1 , \,
    \liminf_{n \to \infty} \left| m(z)-\frac{p_n(z)}{q_n(z)} \right|^{1/n} < 1 .
   \]
\end{corollary}
\begin{proof}
   Using \eqref{def_R} and \eqref{def_L},
   the Liouville-Ostrogradsky formula
   \eqref{Ostrogr2} takes the following form
   \begin{equation} \label{Ostrogr1}
     \alpha_{n}^L(z)q_{n+1}^L(z)r_{n}^R(z)-
     \alpha_n^R(z)q_{n}^L(z)r_{n+1}^R(z)=1.
   \end{equation}
   Since $\sup_n \max\{ |\alpha_{n}^L(z)|,|\alpha_{n}^R(z)| \}
   <\infty$ by assumption on $A,B$, relation~\eqref{residual}
   together with the Cauchy-Schwarz inequality implies that
   \[
          \liminf_{n\to \infty}
          [|q_{n}^L(z)|^2+|q_{n+1}^L(z)|^2]^{1/(2n)} > 1,
   \]
   implying our first claim. The second is established using
   similar techniques, and the third by writing
   $m(z)-p_n(z)/q_n(z)=r_n^L(z)/q_n^L(z)$.
\end{proof}

By having a closer look at the proof, we see that we have
pointwise convergence for a quite dense subsequence, namely for
$p_{n+\epsilon_n}/q_{n+\epsilon_n}$ for $n \geq 0$ with suitable
$\epsilon_n \in \{ 0,1 \}$. We will show in
Theorem~\ref{neighborhoods} below that this point-wise convergence
result can be replaced by a uniform convergence result in
neighborhoods of an element of $\rho(A,B)$.

In the remainder of this subsection we present the main lines of
the proof of Theorem~\ref{thChar}. The first step consists in
showing that our infinite matrix $R(z)$ is a formal left and right
inverse for $zB-A$, compare with~\cite[Section~60 and
Section~61]{wall} for the case of complex Jacobi matrices.

\begin{lemma}\label{GreenPr}
    For any value of $\phi(z)$, the formal matrix products
    $R(z)(zB-A)$ and $(zB-A)R(z)$ give the identity matrix.
\end{lemma}
\begin{proof}
    We will concentrate on the first identity, the second is
    following along the same lines.
    Write shorter
    \[
        q^L_{[0:j]}=(q_{0}^L,\dots,q_{j}^L,0,0,\dots)
    \]
    and similarly $p^L_{[0:j]}$ and $r^L_{[0:j]}$ for
    the row vectors built with the
    other solutions of the recurrence \eqref{rec_L}.
    Then
   \begin{equation}\label{trunc_eqPR}
      p^L_{[0:j]}(z)(zB-A)=-e_0^{\top}+\al_j(z)p_{j+1}^L(z)e_j^{\top}+\ar_jp_j^L(z)e_{j+1}^{\top},
   \end{equation}
   \begin{equation}\label{trunc_eqQR}
      q^L_{[0:j]}(z)(zB-A)=\al_j(z)q_{j+1}^L(z)e_j^{\top}+\ar_jq_j^L(z)e_{j+1}^{\top}.
   \end{equation}
In view of~\eqref{Ostrogr2},~\eqref{def_R}, and~\eqref{def_L}, one
obtains
\[
    (q_j^R(z)p_{[0:j]}^L(z)-p_j^R(z)q_{[0:j]}^L(z))(zB-A)=
     e_j^{\top}-q_j^L(z)e_0^{\top}.
\]
In addition, from \eqref{spectral} and~\eqref{spectral1} we have
that
\[
    (q^L(z) \phi(z) - p^L(z)) (zB-A) = e_0^\top.
\]
A combination of the last two equations shows that, for all $j
\geq 0$,
\[
   ( R(z)_{j,0},R(z)_{j,1},R(z)_{j,2},\dots) (zB-A) = e_j^\top,
\]
as claimed above.
\end{proof}

\begin{proof}[Proof of Theorem~\ref{thChar}.\,]
Let $z\in \rho(A,B)$. Then, according to Lemma~\ref{GreenPr},
$R(z)$ is indeed the matrix representation of the bounded operator
$(zB-A)^{-1}$. We get the decay rate~\eqref{geometric} of the
entries of $R(z)$ from~\cite[Theorem~2.4]{demo} using the fact
that $R(z)$ is the inverse of a bounded tridiagonal matrix.

Suppose now that $\phi(z)\in \mathbb C$ is such
that~\eqref{geometric} is satisfied.
   Then, using the same arguments as in~\cite{AptKV} we have
   that $R(z)$ represents a bounded operator in
   $\ell^2$,
   which by Lemma~\ref{GreenPr}
   is a left and right inverse of $zB-A$. Hence $z\in \rho(A,B)$.
\end{proof}

\begin{remark}\label{rate}
   The essential tool in the proof of Theorem~\ref{thChar} was the
   decay rate \eqref{geometric} of entries of the inverse of a
   bounded tridiagonal matrix. In order to specify the rate of
   convergence, for instance in Corollary~\ref{pointwise},
   it is interesting to quote from~\cite[Theorem~2.4]{demo} possible
   values of $\gamma(z),\delta(z)$ in terms of the condition number
   \[
        \kappa(z)=\| zB-A \| \, \| (zB-A)^{-1}\| \geq 1
   \] being obviously
   continuous in $\rho(A,B)$, compare with
   \cite[Lemma~3.3]{BK97},
   \[
        \delta(z)=\sqrt{\frac{\kappa(z)-1}{\kappa(z)+1}}, \quad
        \gamma(z)=\frac{3\,\| (zB-A)^{-1} \|}{\delta(z)^2}
        \max \Bigl\{ \kappa(z), \frac{(1+\kappa(z))^2}{2\kappa(z)}
        \Bigr\}.
   \]
   \qed
\end{remark}

\subsection{Biorthogonal rational functions and a Favard theorem}\label{sec_2.3}

Our explicit formulas for the entries of the resolvent allow for a
simple proof of biorthogonality for the denominators $q_j^R$ and
$q_k^L$, and in addition an explicit formula for the linear
functional of orthogonality discussed by Ismail and Masson
\cite{IM95}. This generalizes the classical case of $B=I$ and a
selfadjoint Jacobi matrix $A$ \cite{A} where it is well-known
that, for $j\neq k$,
\[
        \langle q_j(A)e_0 , q_k(A) e_0 \rangle = 0.
\]
As a consequence, we obtain a simple proof of the fact that the
$n$th convergent of \eqref{CF} is indeed an $[n-1|n]$th multipoint
Pad\'e approximant of the $m$-function.

In this subsection we denote for $k=0,1,2,...$ by $z_{2k+1}$ (and
by $z_{2k+2}$) the root of $\alpha^L_k$ (and of $\alpha^R_k$,
respectively), where we put $z_{2k+1}=\infty$ (and
$z_{2k+2}=\infty$) if $\alpha_k^L$ (and $\alpha_k^R$) is of degree
$0$. Similar to \cite{IM95} we suppose for convenience that
$z_1,z_2,...\in \rho(A,B)$. More precisely, we suppose that there
exists a domain $\Gamma_{ext}$ with compact boundary forming a
Jordan curve such that
\begin{equation}\label{assumption2}
  z_1,z_2,... \in \Gamma_{ext} \subset \mbox{Clos}(\Gamma_{ext})
  \subset\rho(A,B),
\end{equation}
where $\mbox{Clos}(\cdot)$ denotes the closure.
The case $z_k=\infty$ needs special care: notice that $\infty\in
\rho(A,B)$ if and only if $B$ has a bounded inverse, in which case
we will also suppose that $\infty\in \Gamma_{ext}$. The boundary
$\Gamma$ of $\Gamma_{ext}$ is orientated such that $\Gamma_{ext}$
is on the right of $\Gamma$, implying that
 \[
     g(z) = \frac{1}{2\pi i} \int_\Gamma \frac{g(\zeta)}{z-\zeta}
     d\zeta
 \]
for $z\in \Gamma_{ext}$ and any function $g$ being analytic in
$\rho(A,B)$ and, if $\infty\in \rho(A,B)$, vanishing at infinity.

We start by establishing an integral formula for the entries of
the resolvent.

\begin{lemma}\label{integral}
   Under the assumption \eqref{assumption2}, we have for $z\in \Gamma_{ext}$ and $j,k=0,1,2,...$
   \[
          R(z)_{j,k} = \langle (zB-A)^{-1} e_k,e_j\rangle
          = \frac{1}{2\pi i} \int_\Gamma q_j^R(\zeta)q_k^L(\zeta)
          \frac{m(\zeta)}{z-\zeta} \, d\zeta.
   \]
\end{lemma}
\begin{proof}
    We will consider only the case $j\geq k$, the case $j<k$ is
    similar. Both the resolvent and $R_{j,k}$ are analytic in
    $\rho(A,B)$, and vanishing at infinity provided that $\infty
    \in\rho(A,B)$. Using the explicit formula for $R(z)_{j,k}$
    derived in Theorem~\ref{thChar}, we get for $z\in \Gamma_{ext}$
    \begin{eqnarray*}
          R(z)_{j,k}
          &=&
          \frac{1}{2\pi i} \int_\Gamma r_j^R(\zeta)q_k^L(\zeta)
          \frac{d\zeta}{z-\zeta}
          \\&=&
          \frac{1}{2\pi i} \int_\Gamma q_j^R(\zeta)q_k^L(\zeta)
          \frac{m(\zeta)}{z-\zeta} \, d\zeta -
          \frac{1}{2\pi i} \int_\Gamma p_j^R(\zeta)q_k^L(\zeta)
          \frac{d\zeta}{z-\zeta}.
    \end{eqnarray*}
    It remains to show that the last integral equals zero.
    Denote by $\Omega$ a connected component of
    $\overline{\mathbb C} \setminus \mbox{Clos}(\Gamma_{ext})$.
    If $\Omega$ is bounded, then, by assumption \eqref{assumption2},
    all poles of the rational function $\zeta \mapsto
    p_j^R(\zeta)q_k^L(\zeta)/(z-\zeta)$ are outside of
    $\mbox{Clos}(\Omega)$, and hence the integral over $\partial \Omega$
    is zero. If $\Omega$ is unbounded, then by the above
    assumption on $\Gamma_{ext}$ we may conclude that
    $\infty \not\in \rho(A,B)$, implying that all $z_\ell$ are
    finite. It follows from \eqref{def_R} and \eqref{def_L} that
    all poles of the rational function $\zeta \mapsto
    p_j^R(\zeta)q_k^L(\zeta)/(z-\zeta)$ are outside of
    $\mbox{Clos}(\Omega)$, and this function does vanish at $\infty$.
    Hence again the integral over $\partial \Omega$
    is zero.
\end{proof}

We are now prepared to state and to give a new constructive proof
of the Favard type Theorems \cite[Theorem~2.1 and
Theorem~3.5]{IM95} of Ismail and Masson.

\begin{theorem}\label{favard}
    Under the assumption \eqref{assumption2}, define for
    $g\in \mathcal C(\Gamma)$ the linear functional
    \[
           \sS(g)=\frac{1}{2\pi i} \int_\Gamma g(\zeta) m(\zeta)
           \, d\zeta,
    \]
    then we have the following biorthogonality relations: for any
    $n \geq 1$ and for any polynomial $p$ of degree $<n$ there
    holds
    \[
           \sS\left( q_n^R \frac{p}{\alpha_0^L\alpha_1^L\dots
           \alpha_{n-1}^L} \right) = 0 , \quad
           \sS\left( \frac{p}{\alpha_0^R\alpha_1^R\dots
           \alpha_{n-1}^R} q_n^L  \right) = 0 .
    \]
\end{theorem}
\begin{proof}
    We again only show the first relation, the second follows by
    symmetry. Observe first that $\alpha_{n-1}^L(z_{2n-1})=0$
    implies that $zB-A$ is upper block-diagonal. Since
    $z_{2n-1}\in \rho(A,B)$ by \eqref{assumption2}, we obtain for
    the resolvent $(z_{2n-1}B-A)^{-1}$ the block matrix
    representation
    \[
         \left[\begin{array}{c|c}
               (z_{2n-1}B_{[0:n-1]}-A_{[0:n-1]})^{-1} & *
               \\ \hline
               0 & (z_{2n-1}B_{[n:\infty]}-A_{[n:\infty]})^{-1}
            \end{array}\right].
    \]
    In particular, comparing with \eqref{finite_det} it follows
    that $q_k(z_{2k-1})\neq 0$ for $k=0,1,...,n-1$ (or $\deg
    q_k=k$ provided that $z_{2k-1}=\infty$), and
    \[
         R(z_{2n-1})_{n,k}=0, \quad k=0,1,...,n-1
    \]
    (or $\lim_{z\to \infty} z R(z)_{n,k}=0$ in the case
    $z_{2n-1}=\infty$).
    The first relation implies that
    \[
         \mbox{span}\left\{\frac{q_k^L}{\alpha_{n-1}^L} : k=0,1,...,n-1
         \right\} = \left\{ \frac{p}{\alpha_0^L\alpha_1^L\dots
           \alpha_{n-1}^L} : \deg p < n \right\} ,
    \]
    and the second combined with Lemma~\ref{integral} that
    \[
         \sS \Bigl( q_n^R \frac{q_k^L}{\alpha_{n-1}^L}\Bigr)=0,
         \quad k=0,1,...,n-1,
    \]
    as claimed in Theorem~\ref{favard}.
\end{proof}

\begin{remark}
    In the statement of Theorem~\ref{favard}, one recovers the $m$-function
    as a
    generating function for the linear functional of
    orthogonality,
    since
    \[
          z \mapsto \sS_\zeta\left( \frac{1}{z-\zeta} \right)
          = m(z) , \quad z \in \Gamma_{ext} .
    \]
    Suppose in addition that $\infty \in
    \rho(A,B)$, and thus $B$ has a bounded inverse.
    Then Cauchy's theorem gives the normalisation
    $\sS(1)=m'(\infty)=\langle B^{-1} e_0,e_0 \rangle$, and for
    $\ell
    \geq 0$
    \[
       \sS_\zeta(\zeta^\ell)=\langle B^{-1}(AB^{-1})^\ell e_0,e_0
       \rangle.
    \]
    Similarly, for $z_k\in \Gamma_{ext}$ and $\ell\geq 0$
    we have that
    \[
        \frac{m^{(\ell)}(z_k)}{\ell!}=\sS_\zeta\left(
        \frac{-1}{(\zeta-z_\ell)^{\ell+1}}\right)
        = - \langle B^{-1} (AB^{-1} -z_k)^{-1-\ell} e_0,e_0
       \rangle.
    \]
    Using a partial fraction decomposition, we obtain for any
    polynomial $p$ (of degree $<2n$ if $\infty \not \in \rho(A,B)$)
    the even simpler formula
    \[
          \sS(r)
          = \langle B^{-1} r(AB^{-1}) e_0,e_0 \rangle, \quad
          r = \frac{p}{\alpha_0^L\alpha_0^R\dots
          \alpha_{n-1}^L\alpha_{n-1}^R}.
    \]\qed
\end{remark}

The orthogonality relations of Theorem~\ref{favard} allow now to
show in a simple way that the convergents of our continued
fraction \eqref{CF} are indeed multipoint Pad\'e approximants.

\begin{corollary}\label{interpolants}
     Under the assumption \eqref{assumption2}, for any $n \geq 0$,
     the rational function $p_n/q_n$ is an $[n-1|n]$ multipoint Pad\'e approximant
     of the $m$-function of the pencil $zB-A$ at the points
     $z_1,...,z_{2n}$ counting multiplicities.
\end{corollary}
\begin{proof}
     Relation \eqref{finite_det} shows that $p_n$, $q_n$, are
     polynomials of degree at most $n-1$, and $n$, respectively,
     and from the proof of Theorem~\ref{favard} we know that
     $q_n(z_{2n-1})\neq 0$, hence $q_n$ is non-trivial.

     The interpolation conditions for a Cauchy transform
     (or more generally for a generating function of
     a linear functional) are known to translate to orthogonality
     relations with varying weights, see for instance \cite[Lemma~6.1.2]{STO}.
     Since $r_n^R=(m q_n-p_n)/(\alpha_0^R\dots \alpha_{n-1}^R)$ is
     analytic in $\Gamma_{ext}$ (and vanishes
     at $\infty$ if $\infty\in\Gamma_{ext}$), we only have to
     show that $\omega:=r_n^R/(\alpha_0^L\dots \alpha_{n-1}^L)$
     is analytic in $\Gamma_{ext}$, and, provided that $\infty\in
     \Gamma_{ext}$, its expansion at $\infty$ starts with a term
     $z^{-n-1}$.

     Denote by $\widetilde z_1,...,\widetilde z_{\ell(k)}$ the
     finite points out of $z_1,z_3,...,z_{2k-1}$ with $k \leq n$.
     If $\ell(k)\geq 1$, define
     \[
          \widetilde p(z)=\frac{\alpha_0^L(z)\dots \alpha_{n-1}^L(z)}{(z-\widetilde z_{\ell(1)})\dots (z-\widetilde
          z_{\ell(k)})},
     \]
     a polynomial of degree $<n$. Arguing as in
     Lemma~\ref{integral} and using the Hermite integral formulas
     for divided differences we find that
     \[
           [\widetilde z_{\ell(1)},...,\widetilde z_{\ell(k)}]
             r_n^R
             =
             \frac{1}{2\pi} \int_\Gamma
             \frac{r_n^R(\zeta) \widetilde p(\zeta)}{\alpha_0^L(\zeta)\dots
             \alpha_{n-1}^L(\zeta)} \, d\zeta
             = \sS\left(\frac{q_n^R \widetilde p}{\alpha_0^L\alpha_1^L\dots
           \alpha_{n-1}^L} \right) = 0,
     \]
     where in the last step we have applied the orthogonality
     relation of Theorem~\ref{favard}. Hence $\omega$ is indeed
     analytic in $\Gamma_{ext}$. If $\infty\in \Gamma_{ext}$,
     we find by a similar argument for the expansion of $\omega$
     at $\infty$
     \begin{eqnarray*}
          \omega(z)&=&
             \frac{1}{2\pi} \int_\Gamma
             \frac{r_n^R(\zeta)}{\alpha_0^L(\zeta)\dots
             \alpha_{n-1}^L(\zeta)} \, \frac{d\zeta}{z-\zeta}
             \\
             &=&
             \sS_\zeta \left(\frac{q_n^R(\zeta)}{\alpha_0^L(\zeta)\dots
           \alpha_{n-1}^L(\zeta)(z-\zeta)} \right)
            = \sum_{j=0}^\infty z^{-j-1}
             \sS_\zeta \left(\frac{q_n^R(\zeta) \zeta^j}{\alpha_0^L(\zeta)\dots
           \alpha_{n-1}^L(\zeta)} \right),
     \end{eqnarray*}
     which again by Theorem~\ref{favard} starts with the term
     $z^{-n-1}$.
\end{proof}

\section{Convergence results for multipoint Pad\'e approximants}\label{sec_3}

The aim of this section is to generalize various convergence
results for complex Jacobi matrices to the setting of linear
pencils.

\subsection{Numerical ranges of linear pencils}

It is well known that zeros of formal orthogonal polynomials lie
in the numerical range of the corresponding tridiagonal operator.
Moreover, the corresponding sequence of Pad\'e approximants
converges locally uniformly outside the closure of the numerical
range~\cite[Theorem~3.10]{BK97}. In this section, we generalize
this machinery to the case of linear pencils and multipoint Pad\'e
approximants.

Let us recall that, for a bounded operator $T$ acting in $\ell^2$,
its numerical range is defined by
\[
\Theta(T):=\{(Ty,y)_{\ell^2}: \Vert y\Vert=1\}\subset\dC
\]
Clearly, $\Theta(T)$ is a bounded set. By the Hausdorff theorem we
have that the spectrum $\sigma(T)$ of $T$ is a subset of the
convex set $\overline{\Theta(T)}$ (for instance,
see~\cite[Section~26]{Mar}). The following definition generalizes
the concept of numerical ranges to the linear pencil case.
\begin{definition}[\cite{Mar}]\label{num_ran}
The set
\[
W(A,B):=\{z\in\dC: \langle(zB-A)y,y\rangle_{\ell^2}=0 \text{~for
some } y\ne 0\}
\]
is called a numerical range of the linear pencil $zB-A$.
\end{definition}
The following proposition is immediate from Definition~\ref{num_ran}.
\begin{proposition}\label{3.2}
All the zeros of $q_n$ and $p_n$ belong to $W(A,B)$.
\end{proposition}
\begin{proof}
Let us suppose that $\xi$ is a zero of the polynomial $q_n$. Thus,
according to~\eqref{finite_det}, there exists an element
$y_{\xi}\in\dC^{n}$ such that
\[
(\xi B_{[0:n-1]}-A_{[0:n-1]})y_{\xi}=0,\quad \Vert y_{\xi}\Vert=1.
\]
The latter relation implies $\xi\in
W(A_{[0:n-1]},B_{[0:n-1]})\subset W(A,B)$. Similarly, we have the
inclusion of the zeros of $p_j$ to $W(A,B)$.
\end{proof}
In general, for the bounded operators $A$ and $B$, the set
$W(A,B)$ is neither convex nor bounded.
 However, it turns out that the condition
\begin{equation}\label{NR0}
0\not\in\overline{\Theta(B)}
\end{equation}
implies $\sigma(A,B)\subset\overline{W(A,B)}$~\cite[Section
26]{Mar}, as well as the representation
\begin{equation}
\label{num_ran2}
  W(A,B)=
 \left\{\frac{\langle Af,f\rangle}{\langle Bf,f\rangle}:\,\,
      f\ne 0\right\}=
      \left\{\frac{\langle Af,f\rangle}{\langle Bf,f\rangle}:\,\, \Vert
      f\Vert=1\right\},
\end{equation}
from which we see
the boundedness of $W(A,B)$.
  Condition \eqref{NR0} implies that $B$ is boundedly invertible,
  but, in contrary to the spectrum,
  in general it does not imply any link with the numerical range of
  the operators $B^{-1} A$ or $A B^{-1}$. To see this, one may extend the following simple
  $2\times 2$ example to the pencil case
  $$
       A = \left[\begin{array}{cc} 1 & 1 \\ 0 & 1
              \end{array}\right],
       \qquad B = \left[\begin{array}{cc} 1 & 0 \\ 0 & 2
              \end{array}\right],
  $$ where
  $\Theta(B^{-1} A) \varsubsetneq \Theta(B^{-1/2} A
  B^{-1/2})=\Theta(A,B) \varsubsetneq \Theta(A
  B^{-1})$ are ellipses with the same foci but different
  eccentricities.

Generalizing \cite[Theorem~3.10]{BK97} for complex Jacobi
matrices, we are able to prove a result on locally uniform
convergence which in some sense generalizes the Gonchar
theorem~\cite{Gon82}.
\begin{theorem}\label{un_loc_con}
Let~\eqref{NR0} be satisfied. Then the sequence of multipoint
 Pad\'e approximants $m_{[0:n]}:=p_{n+1}/q_{n+1}$
converges to the $m$-function locally uniformly in $\mathbb C
\setminus \overline{W(A,B)}$.
\end{theorem}

\begin{proof}
   Denote by $D\subset \mathbb C \setminus
    \overline{W(A,B)}$ a closed set with compact boundary. Setting
     $d:=\inf\limits_{\Vert f\Vert=1}|\langle Bf,f\rangle|>0$, we find for
     $z\in \partial D$ and $\Vert f \Vert=1$ that
     \[
\Vert(zB-A)f\Vert\ge 
|\langle
Bf,f\rangle| \, \left| z - \frac{\langle Af,f\rangle}{\langle
Bf,f\rangle} \right| \ge  d \, \dist(z,W(A,B)),
\]
implying that
\[
    \max_{z \in \partial D} \| (zB-A)^{-1} \| \leq d_1 :=
      \frac{1}{d} \max_{z \in \partial D}
      \frac{1}{\dist(z,W(A,B))}.
\]
 Since $W(A_{[0:n-1]},B_{[0:n-1]})\subset W(A,B)$, the same argument
 can be used to estimate the norm of the resolvent of finite subsections
\begin{equation}\label{unbound}
    \max_{z \in \partial D} \| (zB_{[0:n]}-A_{[0:n]})^{-1} \| \leq d_1 .
\end{equation}
Let $\psi$ be  a finite sequence, that is, $\psi=(\psi_1,\dots,\psi_k,0,0,\dots)^{\top}$.
Then
\[
(zB-A)\psi=(zB_{[0:j]}-A_{[0:j]})\psi=\eta
\]
for sufficiently large $j\in\dZ_+$ and $\eta$ is also a finite sequence. Further,
one obviously has
\begin{equation}\label{srcon}
(zB-A)^{-1}\eta=\lim_{j\to\infty}(zB_{[0:j]}-A_{[0:j]})^{-1}\eta.
\end{equation}
Since $zB-A$ is bounded and boundedly invertible, the set of such
$\eta$'s is dense in $\ell^2$
and, therefore, due to~\eqref{unbound} we have that
formula~\eqref{srcon} is also valid for all
$\eta\in\ell^2$
implying the pointwise convergence
$m_{[0:j]}(z)\to m(z)$ for any
$z\in\dC\setminus\overline{W(A,B)}$. Now, the statement of the
theorem immediately follows from~\eqref{unbound} and the Vitali
theorem \cite[Section~5.21]{T}.
\end{proof}

Notice that the concept of a numerical range is valid for
operator-valued functions~\cite{Mar}. Thus the presented approach
can be also generalized to linear pencils proposed in~\cite{Lopez}.

\subsection{Uniform convergence of subsequences in neighborhoods}

We start by improving the pointwise convergence result of
Corollary~\ref{pointwise} generalizing \cite[Corollary~3]{AptKV}.
It was Ambroladze \cite[Corollaries~3 and~4]{AM} who first
observed that, for real Jacobi matrices, a quite dense subsequence
of convergents of \eqref{CF} converges uniformly in a neighborhood
of any element of the resolvent set. This result has been
generalized in \cite[Theorem~4.4]{B99} to the setting of complex
Jacobi matrices. We follow here the lines of the proof presented
in \cite[Theorem~4.7]{B01} since this allows to deduce in the next
subsection a result of convergence in capacity in bounded
connected components of $\rho(A,B)$.

A central observation in what follows is the following result
which for complex Jacobi matrices may be found in
\cite[Proposition~2.2]{B99}.

\begin{proposition}\label{normality}
   The family of rational functions
   \[
         u_n(z) = \frac{q_n(z)}{q_{n+1}(z)} =\frac{q_n^L(z)}{\alpha_n^L(z) q_{n+1}^L(z)} =\frac{q_n^R(z)}{\alpha_n^R(z) q_{n+1}^R(z)}
   \]
   is normal with respect to chordal metric on $\rho(A,B)$.
\end{proposition}
\begin{proof}
   We only have to show that $u_n$ is equicontinuous on the Riemann
   sphere. By the definition of the chordal metric we find for
   $x,y\in \rho(A,B)$
   \[
        \chi(u_n(x),u_n(y))
        = \frac{\Bigl|\alpha^L_n(x) q_{n+1}^L(x)q_n^R(y)-\alpha^R_n(y)
      q_{n}^L(x)q_{n+1}^R(y)\Bigr|}{
       \Bigl\| [q_n^L(x),\alpha_n^L(x)q_{n+1}^L(x)] \Bigr\| \,
       \Bigl\| [q_n^R(y),\alpha_n^R(y)q_{n+1}^R(y)] \Bigr\|}.
   \]
   In order to minorize the denominator, we write shorter as in the proof
   of Lemma~\ref{GreenPr}
    \[
        q^L_{[0:n]}=(q_{0}^L,\dots,q_{n}^L,0,0,\dots), \quad
        q^R_{[0:n]}=(q_{0}^R,\dots,q_{n}^R,0,0,\dots)^\top
    \]
    and observe that
    \[
      q^L_{[0:n]}(x)
      (xB-A) =\Bigl[
      \underbrace{0,...,0}_n,\alpha^L_n(x)q_{n+1}^L(x),-\alpha^R_n(x)q_{n}^L(x),0,...\Bigr],
    \]
    implying that
    \[
        \| q^L_{[0:n]}(x) \|^2 \leq \| (xB-A)^{-1} \|^2
        \,    (1+|\alpha_n^R(x)|^2) \,
     ( |q_n^L(x)|^2+|\alpha_n^L(x)q_{n+1}^L(x)|^2).
    \]
    Similarly,
    \[
      (yB-A)  q^R_{[0:n]}(y)
     =\Bigl[
      \underbrace{0,...,0}_n,\alpha^R_n(y)q_{n+1}^R(y),-\alpha^L_n(y)q_{n}^R(y),0,...\Bigr]^\top,
    \]
    implying that
    \[
        \| q^R_{[0:n]}(y) \|^2 \leq \| (yB-A)^{-1} \|^2
        \,    (1+|\alpha_n^L(y)|^2) \,
     ( |q_n^R(x)|^2+|\alpha_n^R(y)q_{n+1}^R(y)|^2).
    \]
    Finally,
    \begin{align*} &
      \Bigl(\alpha^L_n(x) q_{n+1}^L(x)q_n^R(y)-\alpha^R_n(y) q_{n}^L(x)q_{n+1}^R(y)\Bigr)
      \\&
      =
      q^L_{[0:n]}(x)
      \bigl[ (xB-A)-(yB-A) \bigr]
      q^R_{[0:n]}(y)
      =
      (x-y)
      q^L_{[0:n]}(x) B q^R_{[0:n]}(y),
    \end{align*}
    and a combination of these findings yields that
    $\chi(u_n(x),u_n(y))$ is bounded above by $|x-y|$ times a
    quantity which can be bounded for $x,y$ lying in compact
    subsets of $\rho(A,B)$.
\end{proof}

We are now prepared to generalize \cite[Theorem~4.4]{B99} to
linear pencils.

\begin{theorem}\label{neighborhoods}
     For any $\xi\in \rho(A,B)$ there exists a closed neighborhood
     $V\subset \rho(A,B)$ and $\epsilon_n\in \{0,1\}$ such that
     $m_{[0:n-1+\epsilon_n]}$ converges to $m$ uniformly in $V$.
\end{theorem}
\begin{proof}
    Let $v_n=u_n$ and $\epsilon_n=0$ if $|u_n(\xi)|<1$, or
    elsewhere $v_n=1/u_n$ and $\epsilon_n=1$. Then
    \[
        | m(z)- m_{[0:n-1+\epsilon_n]}(z) | =
        \Bigl|
        \frac{r^L_{n+\epsilon_n}(z)}{q^L_{n+\epsilon_n}(z)} \Bigr|
        =
        \Bigl|
        \frac{\alpha_n^L(z)^{\epsilon_n}r^L_{n+\epsilon_n}(z)
        \sqrt{1+|v_n(z)|^2}}{\sqrt{|q^L_{n}(z)|^2+|\alpha_n^L(z)q^L_{n+1}(z)|^2}}
        \Bigr| .
    \]
    Using the equicontinuity of the $u_n$ (and thus the $v_n$)
    established in Proposition~\ref{normality}, there exists a
    neighborhood $V$ of $\xi$ such that $|v_n(z)|\leq 2$ for all $z\in
    V$. Applying the Cauchy-Schwarz inequality to
    \eqref{Ostrogr1}, we obtain for $z\in V$ the upper bound
    \begin{align*}
          &|m(z)- m_{[0:n-1+\epsilon_n]}(z)|
          \leq
          \\&
          \sqrt{5} \sqrt{|r^L_{n}(z)|^2+|\alpha_n^L(z)r^L_{n+1}(z)|^2}
          \sqrt{|r^R_{n}(z)|^2+|\alpha_n^R(z)r^R_{n+1}(z)|^2}
    \end{align*}
    and the right-hand side tends to zero with a geometric rate
    according to Remark~\ref{rate}.
\end{proof}

One may construct examples with $B=I$ and  selfadjoint $A$ with
the spectrum $\mathbb C \setminus \rho(A,B)$ consisting of two
intervals being symmetric with respect to the origin $\xi=0$, and
$m_{[0:n-1]}$ has a pole at $\xi$ for all odd $n$. This shows that
we may not expect convergence for a subsequence denser than that
of Theorem~\ref{neighborhoods}.

\subsection{Convergence in capacity}

As explained  already before, in general one may not expect
convergence of $m_{[0:n]}$ to $m$ locally uniformly in $\rho(A,B)$
since there might be so-called spurious poles in $\rho(A,B)$. One
strategy of overcoming the problem of spurious poles is to allow
for exceptional small sets, as done in \cite[Theorem~3.1]{B99} for
complex Jacobi matrices where convergence in capacity is
established. We may generalize these findings for linear pencils,
where again we follow the lines of the alternate proof presented
in \cite[Theorem~4.7]{B01}.

\begin{theorem}\label{capacity}
     Let $V$ be a closed connected subset of $\rho(A,B)$
     with compact boundary,
     then there exist $\epsilon_n\in \{0,1\}$ such that
     $m_{[0:n-1+\epsilon_n]}$ converges to $m$ in capacity in $V$.

     If \eqref{NR0} is satisfied and $V \not\subset \overline{W(A,B)}$
     then we obtain convergence in capacity of the whole
     subsequence.
\end{theorem}
\begin{proof}
     Let again be $v_n=u_n^{1-2\epsilon_n}$ with $\epsilon_n\in \{
     0,1\}$ to be fixed later, and consider the sets
     \[
            V_\epsilon:=\{ z\in V : |v_n(z)|\geq 1/\epsilon \}.
     \]
     The arguments in the proof of Theorem~\ref{neighborhoods}
     show that $m_{[0:n-1+\epsilon_n]}$ converges to
     $m$ uniformly in $V\setminus V_\epsilon$.
     It remains thus to show that the capacity of $V_\epsilon$
     tends to zero for $\epsilon\to 0$.

     We choose $\epsilon_n$ in order to insure that the normal
     family $(v_n)_n$ does not have a partial limit being equal to
     the constant $\infty$ in the connected component of $\rho(A,B)$
     containing $V$: this can be done for instance by choosing a
     fixed $\xi\in V$ and to take $\epsilon_n$ as in
     Theorem~\ref{neighborhoods}, namely $\epsilon_n=0$ if
     $|u_n(\xi)|<1$, and elsewhere $\epsilon_n=1$. However,
     under the assumptions of the second part of the statement,
     by taking $\xi\in V \setminus \overline{W(A,B)}$
     it follows from the proof of Theorem~\ref{un_loc_con}
     that
     \[
         \sup_n |u_n(\xi)| = \sup_n | e_n^\top
         (\xi B_{[0:n]}-A_{[0:n]})^{-1} e_n | < \infty ,
     \]
     and hence here we may take the constant sequence
     $\epsilon_n=0$.

     It is now a well-known fact on normal families
     (see for instance \cite[Lemma~2.4]{B99} or the proof of
     \cite[Theorem~4.7]{B01}) that for normal meromorphic families $(v_n)_n$
     with partial limits different from $\infty$ there exist
     monic polynomials $\omega_n$ of bounded degree
     independent of $n$ such that
     \[
           C :=\sup_n \max_{z\in V} |\omega_n(z) v_n(z) | <
           \infty.
     \]
     This enables us to ensure that
     \[
          V_\epsilon \subset
          \{ z\in V : \frac{C}{|\omega_n(z)|}\geq 1/\epsilon \}
          \subset
                    \{ z\in \mathbb C : |\omega_n(z)|\leq \epsilon C
                    \}.
     \]
     Since the capacity increases for increasing sets, and since
     the capacity of the right-hand lemniscate can be explicitly
     computed to be $(\epsilon C)^{1/\deg \omega_n}$, the assertion is proved.
\end{proof}

\section{Biorthogonal rational functions and bi-diagonal decompositions}\label{sec_4}

In this section we give an operator interpretation of the
Darboux transformations of rational solutions of the
difference equations in question (for the orthogonal polynomials
case see~\cite{BM04}). In other words, we present a scheme for
constructing biorthogonal rational functions. As a special case,
we can construct orthogonal rational functions. Note that more
information about orthogonal rational functions can be found in~\cite{Bulth}.

\subsection{ $LU$-factorizations} Let us try to factorize the linear pencil $zB-A$ as
follows
\begin{equation}\label{LU}
zB-A = L(z) D(z) U(z)
\end{equation}
where $D(z)=\mbox{diag}(d_0(z),d_1(z),....)$ is a diagonal matrix, and
$L$, $U$ are bidiagonal matrices of the forms
\[      L = \left(\begin{array}{ccccc}
     1 & 0 & 0 & \cdots   \\
     -v_0^L & 1 & 0 &   \\
     0 & -v_1^L & 1 & \ddots  \\
     \vdots & \ddots & \ddots & \ddots
     \end{array}\right), \quad
     U = \left(\begin{array}{ccccc}
     1 & -v_0^R & 0 & \cdots  \\
     0 & 1 & -v_1^R & \ddots  \\
     0 & 0 & \ddots & \ddots  \\
     \vdots &   & \ddots & \ddots
     \end{array}\right).
\]
Comparing coefficients gives
\[
      - \alpha_n^L = - v_n^L d_{n} , \quad
      - \alpha_n^R = - v_n^R d_{n}  , \quad
      d_0=\beta_0,\quad
      \beta_n=d_{n} + \frac{\alpha_{n-1}^L
      \alpha_{n-1}^R}{d_{n-1}}.
\] Thus $d_0(z)=q_1(z)/q_0(z)$ by \eqref{r2initial}, and by
recurrence using \eqref{r2rec} one deduces that
\[
d_n(z)=\frac{q_{n+1}(z)}{q_n(z)},\quad v_n^L(z) =
\frac{\alpha_n^L(z)q_{n}(z)}{q_{n+1}(z)}, \quad
       v_n^R(z) = \frac{\alpha_n^R(z) q_{n}(z)}{q_{n+1}(z)}.
\]
Hence, the decomposition \eqref{LU} exists if and only iff
$q_n(z)\neq 0$ for all $n \geq 0$. In particular, from
Proposition~\ref{3.2} we obtain existence of such a factorization
for $z\not\in W(A,B)$.

 The decomposition~\eqref{LU} gives us the
possibility to define Christoffel type transformations.
\begin{proposition}\label{Chr}
   Under assumption~\eqref{assumption2}, let $x_0\in \Gamma_{ext}$
   such that the decomposition (\ref{LU}) exists for $z=x_0$.
   Define for $n \geq 0$ the functions
   rational in $x$
   \[
            Q_n^L(x_0,x)=\frac{q_n^L(x)-v_{n}^L(x_0) q_{n+1}^L(x)}{x_0-x}
            , \,
            Q_n^R(x_0,x)=\frac{q_n^R(x) - v_{n}^R(x_0) q_{n+1}^R(x)}{x_0-x}.
   \]
   Then we have the orthogonality relations
\begin{equation}\label{Chr_Or}
  \frac{1}{2\pi i} \int_{\Gamma}
  Q_j^L(x_0,x)Q_k^R(x_0,x)(x_0-x)m(x)dx=\delta_{j,k}/d_j(x_0),
\end{equation}
where $\delta_{j,k}$ is the Kronecker delta.
\end{proposition}
 \begin{proof}
    Denote by $I_{j,k}$ the expression on the left-hand side of
    \eqref{Chr_Or}. We only consider the case $0 \leq j\leq k$, the
    other case follows by symmetry.
By definition of $\cQ_j^L(x_0,x)$ and $\cQ_k^R(x_0,x)$ and by
Lemma~\ref{integral} we obtain \begin{eqnarray*}  I_{jk}&=&
R(x_0)_{k,j} - v_{k}^R(x_0) R(x_0)_{k+1,j} \\& &-v_{j}^L(x_0)
R(x_0)_{k,j+1} + v_{j}^L(x_0) v_{k}^R(x_0) R(x_0)_{k+1,j+1}.
\end{eqnarray*}
For $j<k$, we may apply Theorem~\ref{thChar} and obtain after
factorization
\[
I_{jk}
=(q_j^L(x_0)-v_j^L(x_0)q_{j+1}^L(x_0))(r_k^R(x_0)-v_k^R(x_0)r_{k+1}^R(x_0))=0
\]
by definition of $v_j^L(x_0)$. If $j=k$, we get slightly different
formulas from Theorem~\ref{thChar} and obtain after some
simplifications
\[
I_{jj}=q_j^L(x_0)(r_j^R(x_0)-v_j^R(x_0)r_{j+1}^R(x_0))=
\frac{q_j(x_0)}{q_{j+1}(x_0)}=\frac{1}{d_j(x_0)},
\]
where in the second equality we have applied \eqref{Ostrogr1}.
\end{proof}
\begin{remark}
Clearly, the functions $\alpha_0^L\dots\alpha_{n}^L
Q_n^L(x_0,\cdot)$ and
$\alpha_0^R\dots\alpha_{n}^RQ_n^R(x_0,\cdot)$ are polynomials of
degree $\leq n$.\qed
\end{remark}
Proposition~\ref{Chr} tells us  that the Christoffel
transformation leads to multiplication of the biorthogonality
measure $m(x)$ by a linear factor $(x_0-x)$. This process can be
repeated. Indeed, after the Christoffel transformation we again
obtain a pair of biorthogonal rational functions satisfying a
generalized eigenvalue equation with a new pair
of the Jacobi matrices $\tilde A$, $\tilde B$~\cite{Zhe_BRF}. We can thus apply the
Christoffel transformation to these new functions factorizing the
linear pencil $x_1 \tilde B - \tilde A$ in a similar way as in
\re{LU}. Then the weight function $m(x)(x_0-x)$ is multiplied by a
linear factor $x_1-x$ with $x_1 \ne x_0$. Repeating this process,
let us introduce the polynomial $\pi_N(x) =(x_0-x)(x_1-x)\dots
(x_{N-1}-x)$ with $x_i \ne x_j$, for $i \ne j$ and construct the
functions \[ Q_n^L(x_0,x_1, \dots, x_{N-1}; x) =
\frac{A_{n,N}^L(x)}{\pi_N(x) B_{n,N}}
 \] where
\begin{eqnarray}\label{Zhed_cor_1}
 && A_{n,N}^L(x) = \det \left[ \begin{array}{cccc} q_n^L(x) &
q_{n+1}^L(x) & \dots & q_{n+N}^L(x) \\ q_n^L(x_0) & q_{n+1}^L(x_0)
& \dots & q_{n+N}^L(x_0) \\ \dots & \dots & \dots & \dots\\
q_n^L(x_{N-1}) & q_{n+1}^L(x_{N-1}) & \dots &
q_{n+N}^L(x_{N-1})\end{array} \right ], \label{det_A}
\\&&
B_{n,N}^L = \det \left [ \begin{array}{cccc} q_{n+1}^L(x_0) &
q_{n+2}^L(x_0) & \dots & q_{n+N}^L(x_0) \\ \dots & \dots & \dots &
\dots\\ q_{n+1}^L(x_{N-1}) & q_{n+2}^L(x_{N-1}) & \dots &
q_{n+N}^L(x_{N-1})\end{array} \right ], \label{det_B}
\end{eqnarray}
and similar expressions for $Q_n^R(x_0,...,x_{n-1};x)$,
$A_{n,N}^R(x)$ and $B_{n,N}^R(x)$. Note that if two or more of the
parameters $x_i$ coincide, say $x_1=x_0$, then we may apply a
simple limiting process leading to appearance of derivatives in
corresponding determinants. Then it is easy to show that these
functions satisfy the biorthogonality relation
\begin{eqnarray} && \nonumber
  \frac{1}{2\pi i} \int_{\Gamma}
  Q_j^L(x_0,x_1, \dots, x_{N-1}; x)Q_k^R(x_0,x_1, \dots, x_{N-1}; x)\pi_j(x) m(x)dx
  \\&& \label{Chr_j_Or}
   =\delta_{j,k}/d_j(x_0,x_1, \dots, x_{j-1}),
\end{eqnarray}
with some constants $d_j(x_0,x_1, \dots, x_{j-1})$. Formulas
\eqref{Zhed_cor_1}-\re{Chr_j_Or} are a direct generalization of the Christoffel
formula for the orthogonal polynomials, see, e.g., \cite[\S
2.5]{Sz}.

\subsection{ $UL$-decomposition} For $z\in \rho(A,B)$, let us find
a decomposition
\begin{equation}\label{UL}
zB-A = U(z) D(z) L(z)
\end{equation}
with a diagonal matrix $D(z)=\mbox{diag}(d_0(z),d_1(z),....)$, and
bidiagonal matrices
\[
     U = \left(\begin{array}{ccccc}
     1 & -u_0^R & 0 & \cdots  \\
     0 & 1 & -u_1^R & \ddots  \\
     0 & 0 & \ddots & \ddots  \\
     \vdots &   & \ddots & \ddots
     \end{array}\right),\quad
   L = \left(\begin{array}{ccccc}
     1 & 0 & 0 & \cdots   \\
     -u_0^L & 1 & 0 &   \\
     0 & -u_1^L & 1 & \ddots  \\
     \vdots & \ddots & \ddots & \ddots
     \end{array}\right).
\]
By comparing coefficients we have
\[
      - \alpha_n^L = - u_n^L d_{n+1} , \,
      - \alpha_n^R = - u_n^R d_{n+1}  , \,
      \beta_n=d_{n} + u_n^L u_n^R d_{n+1} =d_{n} + \frac{\alpha_n^L
      \alpha_n^R}{d_{n+1}}.
\] It turns out that this decomposition is unique after fixing an
arbitrary value for $d_0$. Indeed, let $y_{-1}=d_0,y_0=1$, and
consider $y_n$ defined by the recurrence relation~ (\ref{r2rec}). Then it
follows that
\[
    d_n= \frac{\alpha^L_{n-1}\alpha^R_{n-1}y_{n-1}}{y_n},
 \quad u_n^L = \frac{y_{n+1}}{\alpha_n^Ry_n}
      , \quad u_n^R= \frac{y_{n+1}}{\alpha_n^Ly_n},
\]
where from \eqref{r2initial} we learn that
\begin{equation} \label{y_def}
       y_n(z) = (1-m(z)d_0(z)) q_n(z)+d_0(z) r_n(z)
          =q_n(z)-d_0(z) p_n(z)
\end{equation}
for all $n\geq -1$. Thus the decomposition (\ref{UL}) exists if
and only if $d_0(z)\in \mathbb C$ is chosen such that $y_n(z)\neq
0$ and $\alpha_n^L(z)\alpha_n^R(z)\neq 0$ for all $n \geq 0$. For
the special case $d_0(z)=0$ we may compare with the $LU$
decomposition of the preceding subsection and get $u_n^R=1/v_n^L$
and similarly $u_n^L=1/v_n^R$. Also, for the special case
$m(z)d_0(z)=1$ one may show that $y_n(z)=d_0(z)r_n(z)\neq 0$
provided that $z\not\in W(A,B)$.

Suppose that the above factorization exists for $z=x_0$, and
define the Geronimus type transformations by the following
formulas \[
            \cQ_n^L(x_0,x)=q_n^L(x) - u_{n-1}^R(x_0) q_{n-1}^L(x)
            , \,
            \cQ_n^R(x_0,x)=q_n^R(x) - u_{n-1}^L(x_0) q_{n-1}^R(x),
   \]
   and $\cQ_0^L(x_0,x)=\cQ_0^R(x_0,x)=1$.

\begin{proposition}\label{UL_poly}
   Under assumption~\eqref{assumption2}, let $x_0\in \Gamma_{ext}$, $d_0(x_0)\neq 0$, such that
   the above factorization~\eqref{UL} exists for $z=x_0$.
   Consider for
    $g\in \mathcal C(\Gamma)$ the linear functional
    \[
           \widetilde \sS(g)=\frac{1}{2\pi i} \int_\Gamma g(\zeta)
           \frac{m(\zeta)}{x_0-\zeta}
           \, d\zeta + \Bigl( \frac{1}{d_0(x_0)}-m(x_0) \Bigr) g(x_0),
    \]
    then we obtain for $j \neq k$ the biorthogonality relations
   \begin{equation} \label{Ger_Or}
        \widetilde \sS\Bigl(\cQ_j^R(\cdot,x_0)\cQ_k^L(\cdot,x_0)\Bigr)=0,
        \quad
        \widetilde
        \sS\Bigl(\cQ_j^R(\cdot,x_0)\cQ_j^L(\cdot,x_0)\Bigr)=\frac{1}{d_j(x_0)}.
   \end{equation}
\end{proposition}
\begin{proof}
We only look at the case $j\geq k\geq 0$, the other case follows
by symmetry.  Let us compute the $(j,k)$th entry of the product
$L(x_0) (x_0B-A)^{-1}$ (which formally is perhaps expected to be
equal to the upper triangular matrix $D(x_0)^{-1}U(x_0)^{-1}$ but
turns out to be a full matrix). Using Lemma~\ref{integral}
and observing that $x_0\in \Gamma_{ext}$ we get for $j >0$,
\begin{eqnarray*} &&
    \langle L(x_0) (x_0B-A)^{-1} e_k, e_j \rangle
    \\&=&
    \langle (x_0B-A)^{-1} e_k, e_j \rangle
    - u_{j-1}^L(x_0)
    \langle (x_0B-A)^{-1} e_k, e_{j-1} \rangle
    \\&=& \frac{1}{2\pi i} \int_{\Gamma}
    \cQ_j^R(\zeta,x_0) q_k^L(\zeta) \frac{m(\zeta)}{x_0-\zeta}
    d\zeta.
\end{eqnarray*}
Note that the same conclusion is true for $j=0$. If now
$j>k$, we may rewrite the last expression as
\begin{eqnarray*} &&
    \langle L(x_0) (x_0B-A)^{-1} e_k, e_j \rangle
    = 
    \Bigl( r_j^R(x_0)
    - u_{j-1}^L(x_0) r_{j-1}^R(x_0)\Bigr) q_k^L(x_0).
\end{eqnarray*}
Noticing that $u_{j-1}^L(x_0)=y_j^R(x_0)/y_{j-1}^R(x_0)$ with
$y_n^R=y_n/(\alpha_0^R...\alpha_{n-1}^R)$, we get according to
\eqref{y_def}
\begin{eqnarray*} &&
     r_j^R(x_0)
    - u_{j-1}^L(x_0) r_{j-1}^R(x_0)
    \\&=&
     r_j^R(x_0)-\frac{y_j^R(x_0)}{d_0(x_0)}
    - u_{j-1}^L(x_0) \left( r_{j-1}^R(x_0)
    -\frac{y_{j-1}^R(x_0)}{d_0(x_0)}\right)
    \\&=& \Bigl(m(x_0) - \frac{1}{d_0(x_0)} \Bigr)
    \cQ_j^R(x_0,x_0) .
\end{eqnarray*}
Thus for all $g\in \mbox{span}\{q_k^L:k=0,...,j-1\} = \mbox{span}
\{\cQ^L_k(\cdot,x_0):k=0,...,j-1\}$ we conclude that $\widetilde
\sS(\cQ_j^R(\cdot,x_0) g)=0$, and, by definition of $\widetilde
\sS$ and Theorem~\ref{thChar},
\begin{eqnarray*} &&
     \widetilde \sS(\cQ_j^R(\cdot,x_0)\cQ_j^L(\cdot,x_0))
     =
     \widetilde \sS(\cQ_j^R(\cdot,x_0)q_j^L)
    \\&=&
    \langle L(x_0) (x_0B-A)^{-1} e_j, e_j \rangle
    + \Bigl(\frac{1}{d_0(x_0)} -m(x_0)\Bigr)
    \cQ_j^R(x_0,x_0)q_j^L(x_0)
    \\&=&
    \cQ_j^R(x_0,x_0) \Bigl( r_j^L(x_0)
    + (\frac{1}{d_0(x_0)} -m(x_0))
    q_j^L(x_0)\Bigr)
    \\&=&
    \cQ_j^R(x_0,x_0) \frac{y_j(x_0)}{d_0(x_0)\alpha_0^L(x_0)...\alpha_{j-1}^L(x_0)}
    =\frac{1}{d_j(x_0)},
\end{eqnarray*}
the last claim being evident for $j=0$, and for $j>0$ according to
\eqref{Ostrogr2} and \eqref{y_def} \begin{eqnarray*}
    \frac{\cQ_j^R(x_0,x_0)}{d_0(x_0)\alpha_0^L(x_0)...\alpha_{j-1}^L(x_0)}
    &=&
    \frac{p_j q_{j-1} - p_{j-1} q_j}{y_{j-1} \alpha_0^L...\alpha_{j-1}^L\alpha_0^R...\alpha_{j-1}^R}
    \\&=&
      \frac{1}{y_{j-1}\alpha_{j-1}^L\alpha_{j-1}^R} = \frac{y_j(x_0)}{d_j(x_0)}
\end{eqnarray*}
where for simplicity we have dropped in the intermediate
expression the argument $x_0$.
\end{proof}

\begin{remark}
Formula~\eqref{Ger_Or} means that the (bi)orthogonality measure
$\tilde m(x)$ for the transformed rational functions
$\cQ_j^L(x_0,x),\cQ_k^R(x_0,x)$ consists of a regular part
$m(x)/(x_0-x)$ on $\Gamma$ plus a point mass at $x=x_0$, with the mass
$M_0 =2 \pi i (1/d_0(x_0)-m(x_0))$, where $d_0(x_0)\neq 0$ is a free
parameter. A similar situation occurs in the case of ordinary
orthogonal polynomials, where the additional point mass in the
Geronimus transformation can be freely chosen \cite{BM04}.\qed
\end{remark}

\begin{remark}\label{ULB}
   Proposition~\ref{UL_poly} for $x_0\to \infty$ (after multiplication with $x_0$) has been considered before
   in~\cite[Theorem~2.2]{DZ09}.\qed
\end{remark}

\section{An example}\label{sec_5}

In order to illustrate the above findings and to give a
non-trivial example, we study in this section the properties of a
symmetric linear pencils related to a Markov function of the form
\[
     \varphi(z)=\int_a^b \frac{d\mu(t)}{z-t},
\] with a probability measure $\mu$ with a support included in
some compact real interval $[a,b]$. Here, the entries
$A_{j,k},B_{j,k}$ of the linear pencil $zB-A$ for symmetric
interpolation points \begin{equation} \label{markov_ass}
    z_1=\overline{z_2}, z_3=\overline{z_4},... \in \mathbb C
    \setminus [a,b],\quad\dist(z_j,[a,b])>\delta>0,\quad j\in\dN,
\end{equation}
are obtained by developing $\varphi$ into an even part of a Thiele
continued fraction. Before going into details, we recall from the
beginning of \S\ref{sec_1} the special case of
   interpolation at
   infinity $z_1=z_2=z_3=...=\infty$. Here the expansion of $\varphi$
   into a $J$-fraction generates a pencil $zB-A$ with a
   real Jacobi matrix $A$ and with $B=I$ the identity, and it is known
   that the spectrum of the linear pencil
   $zB-A$ (and thus of $A$) is given by the support of the
   underlying measure $\mu$,
   and the numerical range equals to its convex hull $[a,b]$. The aim of
   this section is to show that these properties remain valid for
   more general sets of interpolation points.

Returning to the task of developing $\varphi$ into the continued
fraction in question, the following result has been shown in \cite[Lemma~3.1
and Remark~3.3]{DZ09}, by making the link with Nevalinna
functions. The proof given in~\cite{DZ09} uses the assumption $|\Im z_j|\ge\delta>0$
and it can be immediately generalized to our setting.

\begin{proposition}\label{thiele}
   Suppose that~\eqref{markov_ass} holds, and that $\mu$ has an infinite number of points of
   increase such that $\varphi$ is not a rational function. Then
   there exist probability measures
   $\mu_0=\mu,\mu_1,\mu_2,...$ such that, for all $j
   \geq 0$,
   \[
           \varphi_j(z)= \frac{1}{zB_{j,j}-A_{j,j} - B_{j+1,j}^2
           (z-z_{2j+1})(z-z_{2j+2})\varphi_{j+1}(z)}
   \]
   with the Markov functions
   \[
        \varphi_j(z)=\int_a^b \frac{d\mu_j(t)}{z-t},
   \]
   and the real numbers
   \[
       B_{j,j} =
       \frac{\int_a^b \frac{d\mu_j(t)}{|z_{2j+1}-t|^2}}
            {\left|\int_a^b \frac{d\mu_j(t)}{z_{2j+1}-t} \right|^2 }
            > 1, \,
       A_{j,j} =
       \frac{\int_a^b \frac{t \, d\mu_j(t)}{|z_{2j+1}-t|^2}}
            {\left|\int_a^b \frac{d\mu_j(t)}{z_{2j+1}-t} \right|^2 }
            , \, B_{j+1,j} = \sqrt{B_{j,j}-1}>0 .
   \]
\end{proposition}

Hence, our Markov function $\varphi$ for the symmetric interpolation
points \eqref{markov_ass} induces a linear tridiagonal pencil
$zB-A$ if we set according to (\ref{linpen}),
\begin{eqnarray*} &&
   \beta_j(z)=zB_{j,j}-A_{j,j}, \quad
   \\&& -\alpha^L_j(z)=zB_{j+1,j}-A_{j+1,j} = B_{j+1,j}(z-z_{2j+1}), \quad
   \\&& -\alpha^R_j(z)=zB_{j,j+1}-A_{j,j+1} =
   B_{j+1,j}(z-\overline{z_{2j+1}})=B_{j,j+1}(z-z_{2j+2}).
\end{eqnarray*}
We collect some elementary properties of this pencil in the
following two propositions.

\begin{proposition}\label{Jacobi_bounded}
Suppose that~\eqref{markov_ass} holds.
Then the above tridiagonal matrices $A,B$ are Hermitian and bounded.
\end{proposition}
\begin{proof}
   It follows from \eqref{markov_ass} and the explicit formulas given in
   Proposition~\ref{thiele} that $A$ and $B$ are hermitian,
   and $B$ is real. In order to show that $B$ is bounded, it is sufficient to
   show that its entries are uniformly bounded, where in our case it is sufficient to consider the
   diagonal ones.
Let us first establish the minorization
\begin{equation}\label{markov_min}
\left|\int_a^b \frac{d\mu_j(t)}{z-t} \right|
         \geq \frac{\dist(z,[a,b])}{\max
         \{|z-a|^2,|z-b|^2\}},\quad z\in\dC\setminus[a,b].
\end{equation}
For a proof of~\eqref{markov_min} we suppose that
   $\Re z\geq (a+b)/2$, the other case is similar.
Since $t \mapsto \Im(1/(z-t))$ does not change sign on $[a,b]$, we get
   \[
      \left| \Im \int_a^b \frac{d\mu_j(t)}{z-t} \right|
        = \int_a^b \left| \Im \frac{1}{z-t} \right| d\mu_j(t)
        \geq \frac{|\Im z|}{|z-a|^2} .
   \]
Hence our claim~\eqref{markov_min} follows provided that $|\Im z|=
   \dist(z,[a,b])$. Otherwise, we have that $\Re(z-t)
   \geq \Re(z-b) >0$ for all $t\in [a,b]$, and hence
   \[
      \left| \Re\varphi_j(z) \right|
        \geq
        \frac{\Re(z-b)}{|z-a|^2} ,
   \]
   and the claim follows by observing that $|z-b|=\dist(z,[a,b])$.

   Combining~\eqref{markov_min} with the definition of $B_{j,j}$ given
   in Proposition~\ref{thiele} we conclude that
   \[
   B_{j,j}\le\frac{\max(|z_{2j+1}-a|^4,|z_{2j+1}-b|^4)}{\dist(z_{2j+1},[a,b])^4}
   \]
   the right-hand side being bounded according to assumption~\eqref{markov_ass}.
   Thus $B$ is bounded.

   Similarly, one shows that the diagonal entries $A_{j,j}$ of $A$ are uniformly bounded.
   In order to discuss the off-diagonal entries of $A$, we choose a fixed point
   $z\in\dC\setminus[a,b]$ having a positive distance from the set of the interpolation points
   $z_j$, and get with the help of Proposition~\ref{thiele}
\begin{multline*}
   |A_{j+1,j}|^2=|A_{j,j+1}|^2=|z_{2j+1}|^2B_{j+1,j}^2\\
\le
   \frac{1}{\varphi_{j+1}(z)}\frac{|z_{2j+1}|^2}{|z-z_{2j+1}|^2}
\left(|zB_{j,j}-A_{j,j}|+\frac{1}{|\varphi_j(z)|}\right)
\end{multline*}
the right-hand side being bounded uniformly for $j\ge 0$ according to~\eqref{markov_min}.
Hence, $A$ is also bounded.
   \end{proof}

\begin{proposition}\label{num_Jacobi}
   Suppose that~\eqref{markov_ass} holds.
Then for all $y=(y_0,y_{1},...)^{\top}\in \ell^2$ there holds
   \begin{equation} \label{minoration}
       \langle By,y \rangle \geq |y_k|^2 \quad \mbox{if
       $y_0=...=y_{k-1}=0$.}
   \end{equation}
   Furthermore, for the numerical range of
   Definition~\ref{num_ran} there holds $W(A,B)\subset [a,b]$.
\end{proposition}
\begin{proof}
   In order to show \eqref{minoration},
   let $y=(y_0,y_1,y_2,...)^{\top}\in \mathbb \ell^2$.
   We write as before $y_{[0:n]}=(y_0,y_1,...,y_n,0,0,...)^{\top}\in \ell^2$,
   and notice that $\langle By_{[0:n]},y_{[0:n]} \rangle \to \langle By,y
   \rangle$ for $n \to \infty$ since $B$ is bounded by
   Proposition~\ref{Jacobi_bounded}. Then for $y_0=...=y_{k-1}=0$
   and $n \geq k$ using the relation $B_{j,j}=1+B_{j+1,j}^2$ we
   get
   \begin{eqnarray*}
       \langle By_{[0:n]},y_{[0:n]} \rangle
       &=& \sum_{j=k}^n B_{j,j} |y_j|^2 + 2 \sum_{j=k}^{n-1}B_{j+1,j}
       \Re(y_{j}\overline{y_{j+1}})
       \\&=& |y_k|^2 + \sum_{j=k}^{n-1} | B_{j+1,j} y_j + y_{j+1} |^2
+B_{n+1,n}^2|y_n|^2
\geq
          |y_k|^2,
   \end{eqnarray*}
   implying that \eqref{minoration} holds. Since also
   $\langle Ay_{[0:n]},y_{[0:n]} \rangle \to \langle Ay,y
   \rangle$ for $n \to \infty$, we get using~\eqref{minoration}
   and~\eqref{num_ran2}
   that $W(A,B)$ is included in the closure of the union of the numerical ranges
   $W(A_{[0:n]},B_{[0:n]})$ of all finite sections.
   Further, observe that
\[
W(A_{[0:n]},B_{[0:n]})=\Theta(B_{[0:n]}^{-\frac{1}{2}}A_{[0:n]}B_{[0:n]}^{-\frac{1}{2}}),
\]
where $B_{[0:n]}^{-\frac{1}{2}}A_{[0:n]}B_{[0:n]}^{-\frac{1}{2}}$ is self-adjoint. Thus,
$W(A_{[0:n]},B_{[0:n]})$ is a convex hull of eigenvalues of the matrix
$B_{[0:n]}^{-\frac{1}{2}}A_{[0:n]}B_{[0:n]}^{-\frac{1}{2}}$ or, equivalently, the zeros
of the polynomial
\[
\det(z-B_{[0:n]}^{-\frac{1}{2}}A_{[0:n]}B_{[0:n]}^{-\frac{1}{2}})
=\det B_{[0:n]}^{-1}\det(zB_{[0:n]}-A_{[0:n]})
=\det B_{[0:n]}^{-1}q_{n+1}(z).
\]
To complete the proof it remains to note that
all $n+1$ roots of $q_{n+1}$, that is, the poles of a rational interpolant of a
Markov function are lying in the open interval $(a,b)$, see \cite[Lemma~6.1.2]{STO}.

 \end{proof}

The positive definiteness of finite sections of $B$ also for not
necessarily bounded $[a,b]$ has been shown already in
\cite[Proposition~4.2]{DZ09}, where the authors also establish
\eqref{minoration}.

Notice that property \eqref{minoration} in general does not imply
that condition \eqref{NR0} is true. However, only the latter
condition allows us to conclude that the spectrum of the pencil
$zB-A$ is included in $[a,b]$. There is a special case where we
may say more.

\begin{theorem}\label{spectrum_jacobi}
   Beside \eqref{markov_ass}, suppose in addition that
   $z_{2j+1}=\overline{z_{2j+2}}\to \infty$ as $j\to \infty$.
   Then the operator $B$ is a compact perturbation of the
   identity, and
   condition \eqref{NR0} holds.

   In particular, the spectrum of $zB-A$ is given by the support
   of the measure $\mu$,
   and, outside the spectrum, $\varphi$ coincides
   with the $m$--function of the linear pencil $zB-A$.
\end{theorem}
\begin{proof}
   We have shown in the proof of Proposition~\ref{Jacobi_bounded}
   that $|A_{j+1,j}|^2=|z_{2j+1}|^2 (B_{j,j}-1) = |z_{2j+1}|^2 B_{j+1,j}^2$
   is bounded for $j \to \infty$, and hence
   \[
        \lim_{j\to \infty} B_{j+1,j} =\lim_{j\to \infty} B_{j,j+1}
        =0, \quad
        \lim_{j\to \infty} B_{j,j} = 1,
   \]
   showing that $B$ is a compact perturbation of the identity,
   and $B$ has its numerical range included in $[0,+\infty)$
   by \eqref{minoration}. Hence, if
   \eqref{NR0} does not hold, then $0$ would be an eigenvalue of
   $B$, with corresponding eigenvector $y\in \mathbb \ell^2, y
   \neq 0$. Inserting this $y$ into \eqref{minoration} with $k$
   such that $y_k \neq 0$ gives a contradiction.

   It follows from the text after \eqref{NR0} together with
   Proposition~\ref{num_Jacobi} that the spectrum $\sigma(A,B)$ of
   the linear pencil $zB-A$ is included in $[a,b]$.
   Also, by construction and Corollary~\ref{interpolants},
   $p_n/q_n$ interpolates both $\varphi$ and $m$ in $z_{2n}$,
   implying that these functions are equal for $z=z_{2n}$ and for
   all $n$, and analytic in $\mathbb C \setminus [a,b]$ including $\infty$.
   Since these points accumulate at $\infty$, we
   conclude that $m=\varphi$ outside $[a,b]$. Finally, the
   inclusion $\mbox{supp}(\mu)\subset \sigma(A,B)$ follows from the
   fact that
   $\varphi$ is not analytic in any domain containing points of
   the support of $\mu$.

   Given $z\in \mathbb C \setminus [a,b]$, by choosing a contour
   $\Gamma$ surrounding $[a,b]$ but not the interpolation points
   $z_j$ nor $z$ we get from Lemma~\ref{integral} the formula
   \[
       \langle (zB-A)^{-1} e_k,e_j \rangle
       = \frac{1}{2\pi i} \int_\Gamma q_j^R(\zeta) q_k^L(\zeta)
       \frac{m(\zeta)}{z-\zeta} d\zeta
       = \int_a^b q_j^R(t) q_k^L(t)
       \frac{d\mu(t)}{z-t} ,
   \]
   where for the second identity we have used the Fubini theorem and the fact
   that $\varphi=m$ on $\Gamma$.
   We denote by $R(z)$ the infinite matrix with entries
   $R(z)_{j,k}$ given by the above right-hand integral, which is
   clearly well defined for any $z$ outside the support of $\mu$.
   From Lemma~\ref{GreenPr} we know that $R(z)$ is a formal left
   and right inverse of $(zB-A)$, and the desired conclusion
   $z\not\in \sigma(A,B)$ follows as in the proof of
   Theorem~\ref{thChar} by showing that $R(z)$ is bounded.

   For this last step, we consider the $UL$
   decomposition of $B$ discussed in Remark~\ref{ULB} and
   in~\cite[Theorem~2.2]{DZ09}:
   let $U$ be an upper bidiagonal matrix with ones on the diagonal, and
   the quantities $B_{j+1,j}$ on the  main upper diagonal, then
   $U$ represents a bounded operator on $\ell^2$ according to
   Proposition~\ref{Jacobi_bounded}. Moreover, we have that $B=U
   U^*$, and, with $B$, also $U$ has a bounded inverse. Hence it
   will be sufficient to show that
   \begin{equation} \label{boundedR}
       | \langle U^* R(z) U y,y\rangle | \leq \frac{\langle y,y
       \rangle}{\dist(z,\mbox{supp}(\mu))}
   \end{equation}
   for all $y=(y_0,y_1,...,y_n,0,0,...)^\top \in \ell^2$ and for
   all $n$. Comparing with Proposition~\ref{UL_poly} we find that
   \[
         \langle U^* R(z) U e_k,e_j\rangle
         = \int_a^b \cQ_j^R(\infty,t) \cQ_k^L(\infty,t)
       \frac{d\mu(t)}{z-t}
   \]
   where
   \[
            \cQ_n^L(\infty,x)=q_n^L(x) - B_{n,n-1} q_{n-1}^L(x)
            , \,
            \cQ_n^R(\infty,x)=q_n^R(x) - B_{n,n-1}  q_{n-1}^R(x),
   \]
   and $\cQ_0^L(\infty,x)=\cQ_0^R(\infty,x)=1$, and finally
   \[
          \frac{1}{2\pi i} \int_\Gamma \cQ_j^R(\infty,\zeta) \cQ_k^L(\infty,\zeta)
       m(\zeta) d\zeta
       = \int_a^b \cQ_j^R(\infty,t) \cQ_k^L(\infty,t)\, d\mu(t)
       = \delta_{j,k} .
   \]
   In addition, since $q_n$ has real coefficients, it also follows from
   \eqref{markov_ass} that, for $t\in \mathbb R$,
   \[
        q_j^R(t)=\overline{q_j^L(t)}, \quad
        \cQ_j^R(\infty,t)=\overline{\cQ_j^L(\infty,t)},
   \]
   implying that
   \begin{eqnarray*}&&
      | \langle U^* R(z) U y,y\rangle | = \left|
          \int_a^b \left| \sum_{j=0}^n y_j \cQ_j^L(\infty,t)
      \right|^2 \frac{d\mu(t)}{z-t} \right|
      \\&& \leq  \frac{1}{\dist(z,\mbox{supp}(\mu))}
          \int_a^b \left| \sum_{j=0}^n y_j \cQ_j^L(\infty,t)
      \right|^2 d\mu(t) = \frac{\langle y,y
       \rangle}{\dist(z,\mbox{supp}(\mu))},
   \end{eqnarray*}
   as claimed in \eqref{boundedR}.
\end{proof}

\begin{remark}
The assumption $z_{2j+1}=\overline{z_{2j+2}}\to \infty$ as $j\to \infty$
is very restrictive and can be relaxed. For instance, if
\[
\limsup_{j\to\infty}\frac{\max(|z_{2j+1}-a|,|z_{2j+1}-b|)}{\dist(z_{2j+1},[a,b])}
<\sqrt[4]{2},
\]
then it follows from the proof of Proposition~\ref{Jacobi_bounded}
that $\sup\limits_{j} B_{j+1,j}<1$. As a consequence, the operator
$U$ from the proof of Theorem~\ref{spectrum_jacobi} and thus
$B$ is a boundedly invertible operator.
This implies that~\eqref{NR0} and hence the second part of
the statement of Theorem~\ref{spectrum_jacobi} is still true.
\qed
\end{remark}

In the setting of Theorem~\ref{spectrum_jacobi}, we may therefore
apply our findings of Theorem~\ref{un_loc_con},
Theorem~\ref{neighborhoods}, or Theorem~\ref{capacity} in order to
study the convergence of the multipoint Pad\'e approximants
towards the Markov function $\varphi$, compare with~\cite[Theorem 6.2]{DZ09}.

Finally, returning to the discussion of Remark~\ref{scaling}
concerning the degrees of freedom of representing multipoint
Pad\'e approximants via linear pencils, it is not difficult to see
that the two linear pencils $zB-A$ and $\Delta D (zB-A) D^{-1}
\Delta$ for diagonal $D,\Delta$ with non-zero diagonal entries
generate the same continued fraction \eqref{CF}. Notice that the
matrix $D$ does not affect the diagonal entries and can be
therefore be considered as to be a balancing factor for the
offdiagonal entries, whereas $\Delta$ allows to scale the entries.
In terms of the  continued fraction \eqref{CF}, a scaling
corresponds to considering an equivalence transformation of
(\ref{CF}), and different normalizations can be found in the
literature concerning the special cases of $J$-fractions,
$T$-fractions or Thiele continued fractions. A balancing, however,
leaves invariant the continued fraction (\ref{CF}) and just
addresses the question how to factorize the products
$\alpha_j^L\alpha_j^R$.

It is always possible to choose a scaling such that the resulting
matrices $A$, $B$ become bounded. However, such a scaling might
produce a matrix $B$ having no longer a bounded inverse, or
satisfying no longer the condition \eqref{NR0}. We also know from
\cite[Theorem~2.3]{BK97} that, for fixed $z$, the balancing which
is best for obtaining $z\in \rho(A,B)$ is the one which makes
$zB-A$ to be complex symmetric (i.e., a complex Jacobi matrix). In
the special case of Theorem~\ref{spectrum_jacobi}, we have chosen
a balancing factor to make $B$ real symmetric, and a scaling such
that $A,B$ are bounded and $B$ has a bounded inverse.

A study of best scaling or balancing for general linear pencils is
beyond the scope of this paper. For future research it might be
interesting to consider a (formal) factorization
$z_0B-A=M_1(z_0)M_2(z_0)$ for some fixed $z_0$ (as done in \S~4)
and to discuss the convergence of multi-point approximants in
terms of spectral properties of $z\mapsto M_1(z_0)^{-1} (zB-A)
M_2(z_0)^{-1}$, since this latter quantity does not depend on
scaling or balancing (but depends on how to choose the factors
$M_j(z_0)$).

\end{document}